\documentclass[a4paper]{amsart}
\usepackage{amssymb}
\usepackage{amsthm}
\usepackage{amsmath,amscd}
\usepackage[mathscr]{euscript}
\usepackage[all]{xy}
\usepackage{color}
\usepackage[dvipsnames]{xcolor}
\usepackage[utf8]{inputenc}
\normalfont
\usepackage[T1]{fontenc}
\usepackage[textwidth=14cm,hcentering]{geometry}
\usepackage[colorlinks=true,linkcolor=blue,citecolor=blue]{hyperref}
\newtheorem{theo}{Theorem}[section]
\newtheorem{lemm}[theo]{Lemma}
\newtheorem{prop}[theo]{Proposition}
\newtheorem{coro}[theo]{Corollary}

\theoremstyle{definition}
\newtheorem{defi}[theo]{Definition}
\newtheorem{cons}[theo]{Construction}

\newtheorem{example}[theo]{Example}

\newtheorem{rem}[theo]{Remark}

\newtheorem*{theo*}{Theorem}

\numberwithin{equation}{section}

\newcommand{\op}{^{\mathrm{op}}}
\newcommand{\on}{\operatorname}
\newcommand{\id}{\mathrm{id}}

\newcommand{\Ho}{\mathrm{Ho}}
\newcommand{\Ker}{\mathrm{Ker}}

\newcommand{\Img}{\mathrm{Im}}

\newcommand{\oper}{\mathcal}

\newcommand{\cat}{\mathrm}
\newcommand{\icat}{\mathbf}
\newcommand{\inn}{\icat{N}} 
\newcommand{\ch}{\cat{Ch}_*}
\newcommand{\ich}{\icat{Ch}_*}
\newcommand{\ichp}{\icat{Ch}_{\geq 0}}
\newcommand{\pS}{\widehat{\cat{S}}}

\newcommand{\WMod}{\mathrm{WMod}}
\newcommand{\WComp}{\mathrm{WComp}}
\newcommand{\iWComp}{\mathbf{WComp}}

\newcommand{\TMod}{\mathrm{TMod}}
\newcommand{\TComp}{\mathrm{TComp}}
\newcommand{\iTComp}{\mathbf{TComp}}

\newcommand{\chpure}[1]{\cat{Ch}_*(gr\Aa)^{#1\text{-pure}}}
\newcommand{\apure}{{\alpha\text{-}pure}}
\newcommand{\dga}[2]{gr^{(#1)}\mathsf{DGA}_{#2}}

\newcommand{\cochp}{\cat{Ch}^{\geq 0}}

\newcommand{\chp}{\cat{Ch}_{\geq 0}}

\newcommand{\alg}{\cat{Alg}}
\newcommand{\ialg}{\icat{Alg}}
\newcommand{\Mod}{\mathrm{Mod}}

\newcommand{\ceil}[1]{\lceil#1\rceil}
\newcommand{\floor}[1]{\lfloor#1\rfloor}

\newcommand{\modh}{\,(\text{mod } h)}

\newcommand{\lra}{\longrightarrow}

\newcommand{\Aa}{\mathcal{A}}

\newcommand{\Cc}{\mathcal{C}}
\newcommand{\Dd}{\mathcal{D}}

\newcommand{\Hh}{\mathcal{H}}

\newcommand{\Mm}{\mathcal{M}}

\newcommand{\Qq}{\mathcal{Q}}

\newcommand{\C}{\mathbb{C}}
\newcommand{\Q}{\mathbb{Q}}
\newcommand{\Z}{\mathbb{Z}}
\newcommand{\R}{\mathbb{R}}

\newcommand{\F}{\mathbb{F}}

\renewcommand{\l}{\ell}

\newcommand{\kk}{\mathbf{k}}

\newcommand{\Joana}[1]{{\color{ForestGreen}{#1}}}

\title{Étale cohomology, purity and formality with torsion coefficients}

\author{Joana Cirici}
\author{Geoffroy Horel}

\address{Departament de Matemàtiques i Informàtica, Universitat de Barcelona\\
Gran Via 585\\
08007 Barcelona, Spain  / Centre de Recerca Matemàtica, Edifici C, Campus Bellaterra, 08193 Bellaterra, Spain}
\email{jcirici@ub.edu}

\address{Université Sorbonne Paris Nord, Laboratoire Analyse, Géométrie et Applications, CNRS (UMR 7539), 93430, Villetaneuse, France.}
\email{horel@math.univ-paris13.fr}

\thanks{
J. Cirici acknowledges the Serra H\'{u}nter Program and the AEI (CEX2020-001084-M and PID2020-117971GB-C22).
G. Horel acknowledges support from the project ANR-16-CE40-0003 ChroK. Both authors
thank the ANR-20-CE40-0016 HighAGT 
}

\begin{document}
\maketitle

\begin{abstract}
We use Galois group actions on étale cohomology to prove results of formality
for dg-operads and dg-algebras with torsion coefficients.
Our theory applies, among other related objects, to the 
dg-operad of singular chains of the
operad of little disks and to the dg-algebra of singular cochains of the configuration space of points in the complex space.
The formality that we obtain is only up to a certain degree, which depends on the cardinality of the field of coefficients.
\end{abstract}

\setcounter{tocdepth}{1}

\tableofcontents

\section*{Important note}

This version of the paper is different from the published one \cite{ciricietale}. In that version, Proposition \ref{Nformaldg-algebras} was incorrect as was explained in a Corrigendum (see \cite{ciricicorrigendum}). The current version of the paper is correct as far as the authors know. Here, we have included an appendix that shows that the simple connectivity assumption in Proposition \ref{Nformaldg-algebras} cannot be removed.

\section{Introduction}

The notion of formality was first introduced in the setting of rational homotopy, in which a topological space $X$ is said to be \textit{formal} if its Sullivan algebra of polynomial forms $\Aa_{pl}^*(X)$ is connected to its cohomology $H^*(X;\Q)$ by a string of quasi-isomorphisms of commutative dg-algebras over $\Q$. In this case, any invariant of the rational homotopy type of $X$ can be computed from the cohomology algebra of $X$. 
This notion may be extended to coefficients in an arbitrary commutative ring $R$, by asking that the complex of singular cochains $C^*(X,R)$ is quasi-isomorphic to its cohomology as dg-algebras. We point out here that in general the singular cochains do not have the structure of a commutative dg-algebra so the question of whether cochains and cohomology are quasi-isomorphic as commutative dg-algebras does not even make sense. Although not strictly commutative, the singular cochains have the structure of an $E_\infty$-algebras and this object is a very strong invariant of the homotopy type of the space by a result of Mandell \cite{mandellcochains}. Therefore, one could also ask whether $C^*(X,R)$ and $H^*(X,R)$ are quasi-isomorphic as $E_\infty$-algebras. However, this almost never happens if $R$ is not a $\Q$-algebra (for instance if $R=\F_p$, formality in the $E_\infty$-context implies that Steenrod operations are trivial). Even though it is not the case that the homotopy type of the dg-algebra $C^*(X,R)$ is a complete invariant of the $R$-homotopy type of $X$, several invariants of $X$ can be computed from it, such as the homology of $\Omega X$ with its Pontryagin products structure or the string topology of $X$ when it is an orientable manifold. Furthermore, when $R$ is a field of characteristic zero,  the above notions of formality are equivalent by a result of Saleh \cite{Saleh}. 

The idea that purity implies formality goes back to Deligne, Griffiths, Morgan and Sullivan, who used the Hodge decomposition to show that compact Kähler manifolds are formal over $\Q$ \cite{DGMS}. Since then, Hodge theory has been used successfully several times to prove formality results over $\Q$ in different settings (see for instance 
\cite{Mo}, \cite{santosmoduli}, \cite{Dupont}, \cite{petersenminimal}, \cite{CiHo}). Using the methods of étale cohomology, Deligne \cite{DeWeil} gave an alternative proof of formality for the $\Q_\l$-homotopy type of smooth and proper complex schemes. However,  
it seems that the full power of Galois actions has not been further exploited to address formality questions especially in the case of torsion coefficients. Note that while Hodge theory is confined to rational coefficients, étale cohomology with torsion coefficients is perfectly well-defined and possesses interesting Galois group actions, making it a very valuable tool to study formality in the torsion case.

The theory of weights on the cohomology of algebraic varieties rests on fundamental ideas of Grothendieck and Deligne, and is strongly influenced by Grothendieck's philosophy of mixed motives. Even though a category of motives satisfying all the desired properties is still out of reach, the theory of weights is well-understood on the rational cohomology of complex algebraic varieties and on the étale cohomology of varieties over finite fields. In the case of complex varieties, the cohomology groups support a mixed Hodge structure which is pure for smooth projective varieties. Likewise, the étale cohomology groups of a variety defined over a finite field are acted on by the absolute Galois group of the field. Deligne \cite{DeWeil} showed that the eigenvalues of the Frobenius are in general Weil numbers and that for smooth and projective varieties, the eigenvalues of the Frobenius action on the $n$-th cohomology groups are Weil numbers of pure weight $n$. Let us mention at this point that the Frobenius also acts on the étale cohomology groups with torsion coefficients which is key to the applications we have in mind. It is also important to note that the purity property holds more generally than just for smooth projective varieties: there are many interesting examples of singular and open varieties whose weights in cohomology turn out to be pure in a more flexible way which we call \textit{$\alpha$-purity}.

The notion of formality makes sense (and has proven to be very useful) in many other algebraic contexts outside dg-algebras, such as operads, operad algebras and symmetric monoidal functors. In this paper, we use the theory of weights in étale cohomology to prove partial results of the type ``purity implies formality'' when such algebraic structures arise from the category of algebraic varieties. Although the methods and conditions become quite technical, the theory has applications to very classical and well-known objects. In particular, we address questions of 
Salvatore and Beilinson on formality with torsion coefficients for configuration spaces and little disks operads respectively. The results that we obtain have potential applications in diverse situations, such as to the study of Massey products for complements of subspace arrangements, to embedding calculus, via the Vassiliev-Goodwillie-Weiss spectral sequence, or to deformation quantization \`{a} la Kontsevich with torsion coefficients.

Let us briefly explain how purity can be used to prove formality on a simple example. Consider the étale cohomology of $X=\mathbb{P}^n_{k}$ where $k=\F_{q}$ is a finite field. In that case, the vector space $H^i_{et}(X_{\overline{k}},\Q_\l)$ is of dimension $1$ if $i\leq 2n$ is even and of dimension zero otherwise. Moreover the Frobenius of $k$ acts by multiplication by $q^{i}$ on the $2i$-th cohomology group. Let us write $A=C^*_{et}(X,\Q_\l)$ the dg-algebra of étale cochains on $X$. The Frobenius also acts on $A$ and we can thus consider the subalgebra $B$ where we only keep the generalized eigenspaces of the Frobenius for the eigenvalues that are powers of $q$. This dg-algebra is quasi-isomorphic to $A$ as the other eigenspaces will not contribute to the cohomology. We can then split $B$ as $B=\oplus_{i \in \Z}B(i)$ where $B(i)$ is the generalized eigenspace for the eigenvalue $q^i$. Observe that the cohomology of $B(i)$ will be concentrated in degree $2i$. Moreover, this decomposition is compatible with the multiplication. We have thus produced a multiplicative splitting of the Postnikov filtration of $B$ which is another way to phrase formality. There is a technical difficulty with this sketch as in general, the dg-algebra $A$ will not be finite dimensional and it does not make sense to decompose it as a sum of generalized eigenspaces. However, this issue can be fixed as we will see.
If we wanted instead to prove formality of $A=C^*_{et}(X,\F_\l)$, the above sketch would also work modulo the fact that $q^i$ can be equal to $1$ in $\F_\l$, since $\F_\l^\times$ is a finite group. Therefore, we have a decomposition $B=\oplus_{i\in\Z/h}B(i)$ where $h$ is the order of $q$ in $\F_\l^\times$. This splitting is insufficient to prove formality in full generality but will imply formality when $h\geq n$.

\medskip 

We now spell out in more detail the results that we prove in this paper. Let $K$ be a $p$-adic field and $\overline{K}$ its algebraic closure. We assume that $\overline{K}$ is embedded in $\C$. 
We will denote by $\cat{Sch}_K$ the category of schemes over $K$ that are separated and of finite type.
For $X\in \cat{Sch}_K$, the Galois action on étale cohomology actually exists at the cochain level and there exists a functorial dg-algebra $C^*_{et}(X_{\overline{K}},\F_\l)$
endowed with an endomorphism $\varphi$ corresponding to the choice of a Frobenius lift. The dg-algebra of \textit{étale cochains} relates to singular cochains as follows. Denote by $X_{an}$ the complex analytic space underlying 
$X_\C=X\times_{K}\C$.
Then we have quasi-isomorphisms 
\[
C^*_{sing}(X_{an},\F_\l)\longleftarrow C_{et}^*(X_{\C},\F_\l)\lra C_{et}^*(X_{\overline{K}},\F_\l),
\]
giving symmetric monoidal natural transformations of functors.

Let $\F_q$ be the residue field of $K$ and denote by $h$ the order of $q$
in $\F_\l^\times$. Let $\alpha$ be a positive rational number, with $\alpha<h$.
We say that $H^n_{et}(X_{\overline{K}},\F_\l)$ is a \textit{pure Tate module of weight $\alpha n$}
if the only eigenvalue of the Frobenius is $q^{\alpha n}$, with $\alpha n\in\mathbb{N}$.
If $\alpha n\notin \mathbb{N}$ 
we impose that $H^n_{et}(X_{\overline{K}},\F_\l)=0$.

To study the homologically graded case, we dualize the functor $C^*_{et}(-,\F_\l)$.
As a result, we obtain a lax symmetric monoidal $\infty$-functor
$C_*^{et}(-,\F_\l)$ of \textit{étale chains} from $\inn(\cat{Sch}_K)$ to the $\infty$-category of chain complexes equipped with an automorphism. This allows us to prove the following result:

\theoremstyle{theorem}
\newtheorem*{I2}{\normalfont\bfseries Theorem \textbf{\ref{maintorsion}}}
\begin{I2}
Let $P$ be an operad in sets and let $X$ be a  $P$-algebra in $\cat{Sch}_K$.  Let $N=\floor{(h-1)/\alpha}$. Assume that for each color $c$ of $P$, the cohomology $H^n_{et}(X(c)_{\overline{K}},\F_\l)$ is a pure Tate module of weight $\alpha n$, for all $n$. Then:
\begin{enumerate}
\item There is an $N$-equivalence $C_*(X_{an},\F_\l)\simeq H_*(X_{an},\F_\l)$ in the $\infty$-category of $P$-algebras in $\ich(\F_\l)$.
\item If $P$ is admissible and $\Sigma$-cofibrant, then $C_*(X_{an},\F_\l)$ is $N$-formal as a dg-$P$-algebra.
\end{enumerate}
\end{I2}

By definition, an \textit{$N$-equivalence} is a map of algebras, which induce isomorphisms in homology only up to degree $N$. 
The condition that the operad $P$ is $\Sigma$-cofibrant and admissible
is a technical condition that is only needed in order to transfer $N$-formality from the $\infty$-category of $P$-algebras 
to the category of $P$-algebras.

We mention a few examples where this theorem applies.
Denote by $\overline{\Mm}_{0,n}$ the moduli space of stable
algebraic curves of genus $0$ with $n$ marked points.
The operations that relate the different moduli spaces
$\overline{\Mm}_{0,n}$ identifying marked points form a cyclic operad $\overline{\Mm}_{0,\bullet}$ in the category of smooth proper schemes over $\Z$, whose cohomology is pure. We deduce that 
the cyclic dg-operad $C_*((\overline{\mathcal{M}}_{0,\bullet})_{an},\F_\l)$ is $2(\l-2)$-formal
(Theorem \ref{dmformal}). This extends the formality of $\overline{\Mm}_{0,\bullet}$ over $\Q$ proved in \cite{santosmoduli} to formality with torsion coefficients.

Using the action of the Grothendieck-Teichmüller group, we can also apply this machinery to the little disks operad $\oper{D}$ even though it is not an operad in the category of schemes. We prove $(\l-2)$-formality for the dg-operad $C_*(\oper{D},\F_\l)$.
A similar result applies to the framed little disks (Theorems \ref{littledisksformal} and \ref{framedlittledisksformal}).
A direct consequence is the formality of the $(\l-1)$-truncated operad $C_*(\oper{D}_{\leq (\l-1)},\F_\l)$,
where  $\oper{D}_{\leq n}$ denotes the truncation of $\Dd$ in arity less than or equal to $n$.
This is optimal, since $C_*(\oper{D}_{\leq \l},\F_\l)$ is not formal (see Remark \ref{remarkoptimal}).
This result potentially paves the way for a version of a Kontsevich Formality Theorem over a field of positive characteristic. Also, the 
above method is applied to higher dimensional little disks operads in the work \cite{BH} of the second author with Pedro Boavida de Brito, leading to applications to embedding calculus.
The only extra ingredient needed is the construction of a Frobenius automorphism on the chains over those operads. We point out that Beilinson, in a letter to Kontsevich was conjecturing that the weight filtration in the triangulated category of motives could be used in order to prove formality of the little disks operad over the integers. Remark \ref{remarkoptimal} shows that this was overly optimistic and indicates that our result might be the closest one can get to answering Beilinson's question.

Theorem \ref{maintorsion} is restricted to non-negatively graded homological algebras.  For cohomological dg-algebras we prove the following:

\theoremstyle{theorem}
\newtheorem*{I3}{\normalfont\bfseries Theorem \textbf{\ref{maindg-algebras}}}
\begin{I3}
Let $X\in \cat{Sch}_K$ be a scheme over $K$.
Assume that for all $n$, $H^n_{et}(X_{\overline{K}},\F_\l)$ is a pure Tate module of weight $\alpha n$.
Then the following is satisfied:
\begin{enumerate}
 \item [(i)] If $\alpha(k-2)/h\notin\mathbb{Z}$ then all $k$-tuple Massey products are trivial in $H^*(X_{an},\F_\l)$.
 \item [(ii)]If $H^i(X_{an},\F_\l)=0$ for all $0<i\leq r$ then $H^n(X_{an},\F_\l)$ contains no non-trivial Massey products
 for all $n\leq \ceil{\frac{hr}{\alpha}+2r+1}$.
  \item [(iii)] If it is simply connected, then the dg-algebra
  $C^*_{sing}(X_{an},\F_\l)$ is $N$-formal, with $N=\floor{(h-1)/\alpha}$.
\end{enumerate}
\end{I3}

The above result applies, for instance, to codimension $c$ subspace arrangements defined over $K$ (see Theorem \ref{formal_subspaces}).
In particular, it applies to configuration spaces of points in $\mathbb{A}^n$. In Theorem \ref{formal_config} we
show that the dg-algebra 
$C_{sing}^*(F_m(\mathbb{C}^d),\F_\l)$ is formal up to a degree that depends on $\l$ and $d$.
This in particular gives full formality of $C_{sing}^*(F_m(\mathbb{C}^d),\F_\l)$ whenever $\l\geq (m-1)d+2$.
These results partially answer a question raised at the end of \cite{Salvatore},
about the degree of obstructions to formality over $\F_\l$. 
As another application, the formality of the dg-algebra $C_{sing}^*(\overline{\Mm}_{0,n+1},\F_\l)$, for all $\l\geq n$, is used by Dotsenko in \cite{Dotsenko} to 
estimate the Betti numbers with torsion coefficients of the free loop spaces $L\overline{\Mm}_{0,n+1}$.

Finally, let us mention some other related work. Ekedahl \cite{Ekedahl} proved that for any prime $\l$, there exists a smooth simply connected complex projective surface $X$ with a non-zero Massey product in $H^3(X,\F_\l)$, which is a well-known obstruction to formality. This shows that in general the question of formality with coefficients in a finite field is much more delicate than in the rational case. More recently, Matei \cite{Matei} showed that for any prime $\l$, there is a complement of hyperplane arrangements $X$ in $\C^3$ with a non-zero triple Massey product in $H^2(X,\F_\l)$. Note that Matei's examples have torsion-free pure cohomology. 
We refer the reader to Remark \ref{rem: Matei} for a discussion of this result in relation to our result. Also, in \cite{Salvatore}, Salvatore initiated the study of formality over arbitrary rings for the configuration spaces $F_m(\R^d)$ of $m$ points in $\R^d$. He showed that if $m\leq d$, then $F_m(\R^d)$ is formal over any ring,
but that $F_m(\R^2)$ is not formal over $\F_2$ when $m\geq 4$. 
Note that in all of the above cases, the corresponding spaces are known to be formal over the rationals.
There is an obstruction theory via Hochschild cohomology developed in \cite{Berg}, 
which allows to deduce formality over $\Z$ from formality over $\Q$ in certain quite restricted situations
with torsion-free Hochschild cohomology. For instance, this gives formality over $\Z$ for complex projective spaces. 
More generally, El Haouari \cite{ElHa2} showed that a finite simply connected CW-complex is formal over $\Q$ if and only if it is formal over $\F_\l$ for all primes $\l$ but a finite number. Since $N$-formality implies formality for sufficiently large $N$, our results specify for which primes $\l$ we do have formality over $\F_\l$.

\subsection*{Outline of the paper}

Let us briefly summarize the structure of this paper. In Section \ref{SecModules} we collect the main definitions on weights and Tate modules. We also introduce Tate complexes as more flexible structures than complexes of Tate modules. In Section
\ref{SecEtale} we construct the functor of étale cochains 
from finite type schemes over $K$ to cochain complexes equipped with automorphisms
and compare it with the singular cochains via Artin's theorem.
The next two sections deal with the homologically graded case.
In Section \ref{SecFormalFunctors} we give a criterion of $N$-formality for 
symmetric monoidal functors from chain complexes endowed with a certain $\Z/h\Z$-grading.
We apply this criterion in Section \ref{SecFormalChains}, to prove our main results 
in their chain version. We then address the cohomologically graded case.
Section \ref{SecFormaldgas} is the analogue of Section \ref{SecFormalFunctors} for the case of cohomological dg-algebras. We study Massey products and 
give a criterion of formality for dg-algebras endowed with a $\Z/h\Z$-graded weight decomposition.  These results are applied in Section \ref{SecFormalCochains} to prove our main results in the cochain setting.

\subsection*{Acknowledgments}
We would like to thank Dan Petersen, Paolo Salvatore and Alberto Vezzani for enlightening conversations. We also thank the anonymous referee for several useful comments, and especially for suggesting a better proof of Theorem \ref{theo : Beilinson for Tate} using locally finite dimensional modules.

\subsection*{Notation}

Throughout this paper, we fix a prime number $p$. Denote by $K$ a $p$-adic field (i.e a finite extension of $\Q_p$) and by  $q=p^k$ the cardinality of its residue field. We also assume that a choice of an embedding of $K$ in $\C$ has been made. It follows that there is a preferred choice of embedding of the algebraic closure of $\overline{K}$ in $\C$, namely we can define $\overline{K}$ as the set of complex numbers that are algebraic over $K$. We denote by $\l$ a prime number $\l\neq p$ and write $h$ for the order of $q$ in the group $\F_\l^\times$. 
All schemes over $K$ will be assumed to be separated and of finite type.

\section{Tate modules }\label{SecModules}
In this section we introduce Tate modules and prove a splitting lemma for such objects. We then introduce Tate complexes as more flexible objects than a complexes of Tate modules and compare their categories. This will be used in Section \ref{SecFormalChains}, where we restrict our study to schemes
whose complex of étale chains lands in the category of Tate complexes.

\subsection{Locally finite dimensional $\kk[\varphi]$-modules}

We start with some general algebraic considerations. Let $\kk$ be a field. We consider the category $\Mod_{\kk[\varphi]}$ of $\kk[\varphi]$-modules. We will typically denote an object of this category as a pair $(V,\varphi)$ where $V$ is a $\kk$-vector space and $\varphi$ is a $\kk$-linear endomorphism of $V$. We say that such a module $(V,\varphi)$ is \textit{locally finite dimensional} if any element of $V$ lives in a finite dimensional subspace of $V$ that is stable under $\varphi$. We denote by $\Mod_{\kk[\varphi]}^{lfd}$ the category of such modules.

The following proposition is elementary.

\begin{prop}\label{prop : filtered colimits}
A $\kk[\varphi]$-module $(V,\varphi)$ is locally finite dimensional if and only if it can be written as a filtered colimit of $\kk[\varphi]$-modules whose underlying vector space is finite dimensional.
\end{prop}

\begin{cons}\label{cons : characteristic spaces}
Any locally finite dimensional $\kk[\varphi]$-module $(V,\varphi)$ admits a decomposition as a direct sum of characteristic subspaces
\[V=\bigoplus_{P\in\mathrm{Spec}(\kk[\varphi])}V_P\]
where $V_P=\bigcup_n\Ker P(\varphi)^n$ (we make the abuse of using the same notation for a prime ideal of $\kk[\varphi]$ and a polynomial generating it). Observe that each space in the above decomposition is fixed by $\varphi$.
\end{cons}

\begin{defi} Let $(V,\varphi)$ be a locally finite dimensional $\kk[\varphi]$-module. An element $\lambda$ of $\overline{\kk}$ is said to be an \emph{eigenvalue} if the morphism 
\[\varphi\otimes_{\kk}\overline{\kk}-\lambda\id:V\otimes_{\kk}\overline{\kk}\to V\otimes_{\kk}\overline{\kk}\]
has a non-trivial kernel. A non-zero element in that kernel is called an \emph{eigenvector}.
\end{defi}

The bialgebra structure on $\kk[\varphi]$ induces a symmetric monoidal structure on $\kk[\varphi]$-modules. Explicitly, given $(V,\varphi)$ and $(V',\varphi')$ two $\kk[\varphi]$-modules, then their tensor product is simply given by $(V\otimes_{\kk}V',\varphi\otimes_{\kk}\varphi')$.

\begin{prop}\label{prop : eigenvalues are products}
Let $(V,\varphi)$ and $(V',\varphi')$ be two locally finite dimensional $\kk[\varphi]$-modules. Then their tensor product is locally finite dimensional. Moreover, the eigenvalues of $\varphi\otimes_{\kk}\varphi'$ are exactly the products of eigenvalues of $\varphi$ and eingenvalues of $\varphi'$.
\end{prop}

\begin{proof}
Using Proposition \ref{prop : filtered colimits}, it is easy to show that locally finite dimensional $\kk[\varphi]$-modules are stable under tensor products. 

If $\lambda$ is an eigenvalue of $\varphi$ and $\lambda'$ is an eigenvalue of $\varphi'$, then pick $x$ such that $\varphi(x)=\lambda x$ and $y$ such that $\varphi'(y)=\lambda' y$. We have $(\varphi\otimes\varphi')(x\otimes y)=\lambda\lambda' (x\otimes y)$.

Conversely, let $\gamma$ be an eigenvalue of $\varphi\otimes \varphi'$. Let $x$ be an eigenvector. We can pick a finite dimensional subspace of $V\otimes V'$ that contains $x$ and is stable under $\varphi\otimes\varphi'$. Enlarging this space if necessary, we can further assume that this subspace is of the form $U\otimes U'$ with $U\subset V$ stable under $\varphi$ and $U'\subset V'$ stable under $\varphi'$ and $U$ and $U'$ finite dimensional. We thus have reduced the situation to the finite dimensional case, for which the result is standard.
\end{proof}

\subsection{Tate modules}

Recall that $q=p^k$ and that $\l\neq p$ is a prime number.

\begin{defi}
Let $(V,\varphi)$ be a locally finite dimensional $\F_\l[\varphi]$-module. We say that the pair $(V,\varphi)$ is a \textit{$q$-Tate module}
if the only eigenvalues of $\varphi$ are powers of $q$. A $q$-Tate module is said to be \textit{pure of weight $n$} if the only eigenvalue is $q^n$.
\end{defi}

\begin{rem}
The name Tate module comes from the fact that those objects are constructed from the Tate twists $\F_\l(n)$ defined below. Tate modules in our sense have little to do with the Tate module of an abelian variety. We hope that this does not lead to any confusion.
\end{rem}

In practice, we will drop the mention of $q$ and just say Tate module. Denote by $\TMod$ the category of Tate modules. Morphisms in these categories are given by morphisms of $\F_\l[\varphi]$-modules.

\begin{example}
We denote by $\F_\l(n)$ the one-dimensional $\F_\l$-vector space equipped with the automorphism $\varphi=q^n\id$. Then $\F_\l(n)$ is a Tate module of pure weight $n$.
\end{example}

\begin{lemm}\label{TModsmc}
The category $\TMod$ is an abelian subcategory of $\Mod_{\F_\l[\varphi]}$. Moreover, the tensor product of $\Mod_{\F_\l[\varphi]}$ restricts to $\TMod$ and is exact in both variables.
\end{lemm}

\begin{proof}
Obviously a submodule of a locally finite dimensional $\F_\l[\varphi]$-module is locally finite dimensional. If $V\subset W$ is an inclusion of locally finite dimensional $\F_\l[\varphi]$-modules, then any eigenvalue of $V$ is an eigenvalue of $W$. It follows that $\TMod$ is stable under taking subobjects.
Now consider a short exact sequence of $\F_\l[\varphi]$-modules
\[0\to V'\to V\to V''\to 0\]
Assume that $V$ is locally finite dimensional. Then for an element $x$ of $V''$, we can pick any lift $\overline{x}$ in $V$. By assumption, $\overline{x}$ lives in a finite dimensional submodule of $V$. The image of that submodule in $V''$ is a finite dimensional submodule that contains $x$. This implies that $V''$ is locally finite dimensional.

Now assume  that $V$ is a Tate module. Let $\lambda\in \overline{\F}_\l$ be an eigenvalue of $V''$. Let us pick a subspace $U''$ of $V''\otimes_{\F_{\l}}\overline{\F}_\l$ that is finite dimensional, stable under the action of $\varphi$ and contains $x$. We can then pick a finite dimensional subspace $U$ of $V\otimes_{\F_{\l}}\overline{\F}_\l$ that is finite dimensional, stable under the action of $\varphi$ and surjects onto $U''$. Let us denote by $U'$ the kernel of the surjection $U\to U''$. We thus have a short exact sequence
\[0\to U'\to U\to U''\to 0\]
of finite dimensional $\overline{\F}_\l[\varphi]$-modules. In this finite dimensional situation, it is classical that all the eigenvalues of $U''$ are eigenvalues of $U$. In particular $\lambda$ is a power of $q$ as desired. It follows that $V''$ must be a Tate module. This concludes the proof that $\TMod$ is an abelian subcategory of $\Mod_{\F_\l[\varphi]}$. 

The fact that $\TMod$ is stable under tensor product follows immediatly from Proposition \ref{prop : eigenvalues are products}. Finally, the exactness of the tensor product is already true in $\Mod_{\F_\l[\varphi]}$.
\end{proof}

\subsection{Deligne splitting}

If $\Aa$ is a symmetric monoidal abelian category, we denote by $gr^{(h)}\Aa$ the category of $\Z/h\Z$-graded objects of $\Aa$. It is a symmetric monoidal category, with the tensor product given by
\[(A\otimes B)^n:=\sum_{a+b\equiv n \modh} A^a\otimes B^b.\]
The functor $U:gr^{(h)}\Aa\lra \Aa$ obtained by forgetting the degree is symmetric monoidal.

The following Lemma is an analogue of Deligne's splitting for mixed Hodge structures in the setting of Tate modules.

\begin{lemm}\label{splittingFrobT}
The functor 
$\Pi:\TMod\lra \Mod_{\F_\l}$  defined 
by $(V,\varphi)\mapsto V$ admits a factorization
\[
\xymatrix@R=4pc@C=4pc{
\TMod\ar[r]^{G}\ar[dr]_{\Pi}&gr^{(h)}\Mod_{\F_\l}\ar[d]^{U}\\
&\Mod_{\F_\l}
}
\]
into symmetric monoidal functors, where 
$G(V,\varphi)^n$ is the space $V_P$ associated to the prime ideal $P=\langle \varphi-q^n\rangle$.
\end{lemm}

\begin{proof}
A priori $U\circ G(V)$ is a subspace of $\Pi(V)$. We will check that the spaces $V_P$ are zero for $P$ a prime polynomial whose roots are not powers of $q$. Let $P$ be such a polynomial and assume that $ V_P\neq 0$. Let $W\subset V_P$ be a non-zero finite dimensional subspace that is stable under $\varphi$. Then the minimal polynomial of the restriction of $\varphi$ to $W$ is a power of $P$. This means that $\varphi$ has an eigenvalue that is a root of $P$ which is absurd. In conclusion, we see that $U\circ G(V)\cong \Pi(V)$.

It remains to check that $G$ is a symmetric monoidal functor. Let $(V,\varphi)$ and $(V',\varphi')$ two locally finite dimensional $\F_\l[\varphi]$-modules. Then, we claim that we have an inclusion
\begin{equation}\label{eqn : inclusion}
G(V\otimes V',\varphi\otimes\varphi')^n\subset \bigoplus_{p+q=n}G(V,\varphi)^p\otimes G(V',\varphi')^q
\end{equation}
Assuming this is true for the moment, we can sum these inclusions over all choices of $n$ and we see that they must all be equalities. This shows that $G$ is indeed symmetric monoidal. 

It remains to prove that inclusion (\ref{eqn : inclusion}) is indeed true. It is straightforward if $V$ and $V'$ are finite dimensional. In general, let us pick an element $x$ of $G(V\otimes V',\varphi\otimes\varphi')^n$. Then $x$ lives in a subspace $U\otimes U'$ of $V\otimes V'$ with $U\subset V$ and $U'\subset V'$ two finite dimensional $\F_\l[\varphi]$-submodules. This reduces the situation to the finite dimensional case.
\end{proof}

\subsection{Tate complexes}

We will consider more flexible structures than the ones provided by chain complexes of Tate modules.

\begin{defi}
A \textit{Tate complex} is a chain complex $(C,\varphi)$ in $\Mod_{\F_\l[\varphi]}^{lfd}$ whose homology is a Tate module in each degree.
\end{defi}

We denote the category of Tate complexes by $\TComp$. The tensor product of locally finite dimensional $\F_\l[\varphi]$-modules induces a symmetric monoidal structure on this category.

We have an obvious symmetric monoidal inclusion
\[\iota: \ch(\TMod)\to\TComp\]

\begin{theo}\label{theo : Beilinson for Tate}
There exists a symmetric monoidal functor
\[\pi:\TComp\to \ch(\TMod)\]
such that 
\begin{enumerate}
\item The composite $\pi\circ\iota$ is naturally isomorphic to the identity.
\item The composite $\iota\circ\pi$ is naturally quasi-isomorphic to the identity.
\end{enumerate}
\end{theo}

\begin{proof}
For a Tate complex $C$, we define
\[\pi(C_*)=\bigoplus_{P\in W } (C_*)_P\] where $W$ is the set of prime polynomials whose roots are powers of $q$. Clearly, this summand of $C_*$ is a subcomplex and is functorial in $C_*$. It is also not hard to check that $\pi$ is a symmetric monoidal functor (this is a variant of the proof of Lemma \ref{splittingFrobT}). From this description, it is obvious that $\pi\circ\iota$ is isomorphic to the identity. On the other hand, for a Tate complex, we have a splitting of complexes $C_*=\pi(C_*)\oplus R$. The complex $R$ has only eigenvalues that are not powers of $q$. It follows that it is also the case for its homology. Since $C_*$ is a Tate complex, this forces $R$ to be acyclic. This concludes the proof.
\end{proof}

\section{Functor of étale cochains}\label{SecEtale}

In this section, we define the functor of étale cochains from the category $\cat{Sch_K}$ of schemes over $K$
that are separated and of finite type to the category of Tate complexes. The dg-algebra of étale cochains may be defined in the context of sheaf theory, using the global sections functor $\mathbb{R}\Gamma(X_{et},-)$
where $X_{et}$ is the étale site of a scheme $X$. There is a model for
$\mathbb{R}\Gamma(X_{et},\F_\ell)$ that is an $E_\infty$-algebra by construction  (see for instance \cite{Petersencompo}, \cite{ChCiSheaves} and \cite{PetersenRemark} for a treatment of sheaves of $E_\infty$-algebras on topological spaces and \cite{RoigSite} for multiplicative sheaves on a site). However, we have made the choice of using an alternative approach to the \'etale cochain functor, using the theory of the \'etale homotopy type of Artin-Mazur and Friedlander (see \cite{ArtinMazur,friedlanderetale}). This construction associates to a scheme a profinite simplicial set. We can then apply the singular cochain functor and obtain an $E_{\infty}$-algebra of \'etale cochains. This detour might seem slightly unnatural to an algebraic geometer. It is indeed true that all our examples that come from algebraic geometry could be dealt with using the more classical sheaf-theoretic approach. However, one of our main example of applications is the little $2$-disks operad (Theorem \ref{littledisksformal}). In that case, we are given an operad in profinite simplicial sets with an action of the Grothendieck-Teichm\"uller group. The advantage of the approach that we have chosen in this paper is that this example and the algebro-geometric example will be treated in a completely parallel fashion.


Recall that a profinite set is a compact Hausdorff totally disconnected topological space. Alternatively, this is a pro-object in the category of finite sets. We denote by $\pS$ the category of profinite spaces. Its objects are simplicial profinite sets or alternatively pro-objects in the category of simplicial sets that are degreewise finite and coskeletal \cite[Proposition 7.4.1]{barneapro}. 

Given a profinite set $X$ and a finite commutative ring $R$, we denote by $R^X$, the ring of continuous maps $X\to R$ where $R$ is given the discrete topology.  If $X$ is a profinite space, we denote by $S^\bullet(X,\F_\l)$ the cosimplicial commutative $\F_\l$-algebra given by $\F_\l^{X_n}$ in degree $n$. We also denote by $C^*(X,\F_\l)$ the result of the application of the Dold-Kan equivalence to $S^\bullet(X,\F_\l)$. By the fact that the Dold-Kan construction is lax monoidal, we deduce that the resulting object is a dg-algebra. Note however, that it is not a commutative dg-algebra (instead it is naturally an algebra over the Barrat-Eccles operad, see \cite{bergercombinatorial}).

Given a profinite space $X$, we denote by $H^*(X,\F_\l)$, the cohomology of $C^*(X,\F_\l)$ which is naturally a commutative graded algebra over $\F_\l$.

In this context, we also have a K\"unneth formula given by the following proposition.

\begin{prop}\label{prop : pro Kunneth}
Let $X$ and $Y$ be two profinite spaces. The canonical map
\[S^\bullet(X,\F_\l)\otimes S^\bullet(Y,\F_\l)\to S^\bullet(X\times Y,\F_\l)\]
is an isomorphism. In particular, there is a Künneth isomorphism
\[H^*(X,\F_\l)\otimes H^*(Y,\F_\l)\cong H^*(X\times Y,\F_\l)\]
\end{prop}

\begin{proof}
It suffices to check it in each cosimplicial degree. Thus it is enough to prove that for $X$ and $Y$ two profinite \emph{sets}, the canonical map
\[(\F_\l^X)\otimes (\F_\l^Y)\to \F_\l^{X\times Y}\]
is an isomorphism. Since the tensor product commutes with filtered colimits, we can assume that $X$ and $Y$ are finite sets in which case the statement is straightforward.
\end{proof}

\begin{defi}\label{defi l complete equivalence}
We say that a map $X\to Y$ of profinite spaces is an \textit{$\l$-complete equivalence} if the induced map
$H^*(Y,\Z/\l)\to H^*(X,\Z/\l)$
is an isomorphism.
\end{defi}

Recall from Quick \cite{quickcontinuous} that there exists a functor $\on{Et}$ from schemes over $K$ to profinite spaces equipped with a continuous action of the absolute Galois group of $K$. Since $K$ is a $p$-adic field with residue field $\F_q$, there is a surjective map
\[\on{Gal}(\overline{K}/K)\to \on{Gal}(\overline{\F}_q/\F_q).\]
The target of this map is isomorphic to $\widehat{\Z}$ generated by the Frobenius. We make once and for all a choice of a lift of the Frobenius in $\on{Gal}(\overline{K}/K)$ and this defines a continuous morphism
\[\widehat{\Z}\to \on{Gal}(\overline{K}/K).\]
We can restrict the Galois action on $\on{Et}(X)$ along this morphism we get a functor 
\[\on{Et}:\cat{Sch}_K\lra \pS^{\widehat{\Z}}.\]
This functor has the property that there is an isomorphism 
\[H^*(\on{Et}(X),\F_\l)\cong H^*_{et}(X_{\overline{K}},\F_\l)\] where the $\Z$-action on the left-hand side corresponds to the action of the chosen Frobenius lift on the right-hand side.

\begin{prop}\label{prop: Etale Kunneth}
Let $X$ and $Y$ be two $K$-schemes of finite type. The canonical map
\[\on{Et}(X\times Y)\to \on{Et}(X)\times\on{Et}(Y)\]
is an $\l$-complete equivalence.
\end{prop}

\begin{proof}
In light of Proposition \ref{prop : pro Kunneth} and of the identification in the paragraph just above, we have to prove that the canonical map
\[H^*_{et}(X_{\overline{K}},\F_\l)\otimes H^*_{et}(Y_{\overline{K}},\F_\l)\to H^*_{et}((X\times Y)_{\overline{K}},\F_\l)\]
is an isomorphism. This is just the Künneth theorem for étale cohomology \cite[Cor. 1.11]{sga4.5}.
\end{proof}

\begin{cons}
We define the functor $C^*_{et}(-,\F_\l)$ on finite type schemes over $K$ as the following composition
\[\cat{Sch}_K\xrightarrow{\,\on{Et}\,}\pS^{\Z}\xrightarrow{C^*(-,\F_\l)}(\ch(\Mod_{\F_\l[\varphi]}^{lfd}))\op\]
This lands in $\ch(\Mod_{\F_\l[\varphi]}^{lfd})$ and not merely in $\ch(\Mod_{\F_\l[\varphi]})$ because of Proposition \ref{prop : filtered colimits} and the observation that for a profinite set $X$, the vector space $\F_\l^X$ is a filtered colimit of finite dimensional vector spaces. 
\end{cons}

\begin{prop}\label{etalproduct}
This functor is oplax monoidal and the natural transformation
\[C^*_{et}(X,\F_\l)\otimes C^*_{et}(Y,\F_\l)\to C^*_{et}(X\times Y,\F_\l)\]
is a quasi-isomorphism.
\end{prop}

\begin{proof}
The first claim follows by observing that this functor is obtained by composing two oplax monoidal functors. The second claim follows by combining Proposition \ref{prop: Etale Kunneth} and Proposition \ref{prop : pro Kunneth}. 
\end{proof}

\begin{cons}\label{cons: etale cochains}
This proposition implies that the functor $C^*_{et}(-,\F_\l)$ descends 
to a symmetric monoidal $\infty$-functor
\[\inn(\cat{Sch}_K)\op\lra\ich(\Mod_{\F_\l[\varphi]}^{lfd})\]
that we still denote by $C^*_{et}(-,\F_\l)$.
\end{cons}

We will need to compare the étale chains and cochains to singular chains and cochains. For $X$ a scheme over $K$, we denote by $X_{an}$ the complex analytic space underlying $X_\C$ (recall that $\overline{K}$ comes with a preferred embedding into $\C$).

\begin{prop}\label{prop : comparison}
There is an equivalence 
\[C^*((-)_{an},\F_\l)\simeq C^*_{et}(-,\F_\l)\]
in the category of symmetric monoidal $\infty$-functors from $\inn(\cat{Sch}_K)$ to $\ich(\F_\l)$.
\end{prop}

\begin{proof}
We construct a zig-zag of symmetric monoidal natural transformation of functors from schemes over $K$ to cochain complexes connecting $C^*_{sing}((-)_{an},\F_\l)$ and $C^*_{et}(-,\F_\l)$. The middle term in this zig-zag is the functor $C^*_{et}((-)_{\C},\F_\l)$. There is a map 
\[C^*_{et}(X_{\C},\F_\l)\to C^*_{sing}(X_{an},\F_\l).\]
This is the comparison between étale and singular cohomology which is a quasi-isomorphism for any scheme $X$
(see \cite[XVI, Theorem 4.1]{sga4} or \cite[Theorem 8.4]{friedlanderetale} for \'etale homotopy types). There is also a base change map
\[C^*_{et}(X_{\C},\F_\l)\to C^*_{et}(X_{\overline{K}},\F_\l)=C^*_{et}(X,\F_\l)\]
which is also a quasi-isomorphism for any scheme $X$
(see  \cite[XII, Corollary 5.4]{sga4} for a proof in the proper case, see 
\cite[p. 231, Corollary 4.3]{Milne} for the general case).
\end{proof}

\section{Formality criteria for symmetric monoidal functors}\label{SecFormalFunctors}

Let $\Aa$ be a symmetric monoidal abelian category whose tensor product is exact in both variables. In this section, we give a formality criterion for symmetric monoidal functors of chain complexes of objects of $\Aa$ equipped with an extra grading.

\begin{defi}
Let $N$ be an integer. A morphism of chain complexes $f:A\to B \in\ch(\Aa)$ is called \textit{$N$-quasi-isomorphism} if the induced morphism in homology $H_i(f):H_i(A)\to H_i(B)$ is an isomorphism for all $i\leq N$.
\end{defi}

\begin{defi}\label{defFormalFunctor}
Let $\Cc$ be a symmetric monoidal category and
$F:\Cc\lra \ch(\Aa)$ a lax symmetric monoidal functor.
Then $F$ is said to be a \textit{formal} (resp. \textit{$N$-formal}) \textit{lax symmetric monoidal} functor if 
there is a string of monoidal natural transformations of lax symmetric monoidal functors 
\[F\stackrel{\Phi_1}{\Longleftarrow}F_1\Longrightarrow \cdots \Longleftarrow F_n\stackrel{\Phi_n}{\Longrightarrow}H_*\circ F\] such that for every object $X$ of $\Cc$, the morphisms $\Phi_i(X)$ are quasi-isomorphisms (resp. $N$-quasi-isomorphisms).
\end{defi}

\begin{defi}
Let $\Cc$ be a symmetric monoidal category and $F:\inn(\Cc)\to\ich(\Aa)$ a lax symmetric monoidal functor
(in the $\infty$-categorical sense). We say that $F$ is a \textit{formal} (resp. \textit{$N$-formal}) \textit{lax symmetric monoidal $\infty$-functor} if $F$ and $H_*\circ F$ are quasi-isomorphic (resp. $N$-quasi-isomorphic) as lax monoidal functors from $\inn(\Cc)$ to $\ich(\Aa)$.
\end{defi}

\begin{rem}\label{rem: N formality}
Denote by $t_{\leq N}:\chp(\Aa)\lra \chp(\Aa)$ the truncation functor defined by
\[(t_{\leq N}A)_n:=\left\{ 
\begin{array}{ll}
A_n& \mathrm{if}\; n<N\\
A_N/\on{\Img}(d_{N+1})&\mathrm{if}\; n=N\\
0&\mathrm{if}\; n>N
\end{array}\right..\]
This functor can be checked to be lax monoidal. Since it preserves quasi-isomorphisms, it descends to a lax monoidal $\infty$-functor $\ichp(\Aa)\to \ichp(\Aa)$. Moreover, for any lax monoidal functor $F:\Cc\to\chp(\Aa)$ the canonical map $F\to t_{\leq N}F$ is an $N$-quasi-isomorphism of lax monoidal functors. Likewise in the $\infty$-categorical case, given $F:\inn(\Cc)\to\ichp(\Aa)$ a lax monoidal $\infty$-functor, the canonical map $F\to t_{\leq N}F$ is an $N$-quasi-iomorphism of lax monoidal $\infty$-functors. From this observation, we deduce that, in both cases, $F$ is $N$-formal if and only if $t_{\leq N}F$ is formal.
\end{rem}

\begin{rem}\label{rem : cohomological grading}
The previous remark fails if we use the cohomological grading. In that case, we can define $N$-formality of an algebraic structure in cochain complexes as being a zig-zag of maps that induce isomorphisms in cohomology up to degree $N$. However, this notion cannot be interpreted as usual formality of the truncation since in the cohomological case, the truncation will not be lax monoidal (it is in fact oplax monoidal).
\end{rem}

A formal functor immediatly yields formality results for operads by the following Proposition.

\begin{prop}
Let $\Cc$ be a symmetric monoidal category and $P$ be an operad in sets.
\begin{enumerate}
\item Let $F:\Cc\lra \chp(\mathcal{A})$ be a formal (resp. $N$-formal) lax symmetric monoidal functor. Then for any $P$-algebra $X$ in $\Cc$, the dg-$P$-algebra $F(X)$ is formal (resp. $N$-formal) in $\Cc$.
\item Let $F:\inn(\Cc)\lra \ichp(\mathcal{A})$ be a formal (resp. $N$-formal) lax symmetric monoidal functor. Then for any $P$-algebra $X$ in $\Cc$, we have an equivalence $F(X)\simeq H(F(X))$ (resp. $t_{\leq N}F(X)\simeq t_{\leq N}H(F(X))$ in the $\infty$-category $\ialg_P(\ich(\kk))$.
\end{enumerate}
\end{prop}

\begin{proof}
This is immediate. See for instance \cite[Proposition 2.5.5]{santosmoduli}.
\end{proof}

\begin{rem}
Formality in the $\infty$-categorical sense is in general weaker than formality in the model categorical sense. Our paper gives a method for proving formality in the $\infty$-categorical sense. When the model structure on $\alg_P(\ch(\kk))$ happens to model the $\infty$-category $\ialg_P(\ich(\kk))$, this implies formality in the model categorical sense. But in the other cases, the model category $\alg_P(\ch(\kk))$ does not model the correct homotopy theory so the question of having formality is not particularly interesting. In any case, we have the following rigidification result due to Hinich that gives conditions on $P$ under which the two notions are equivalent.
\end{rem}

\begin{prop}\label{prop: rigidification positive char}
Let $\kk$ be a field. Let $P$ be an operad in sets that is admissible and $\Sigma$-cofibrant. Let $A$ be a $P$-algebra in $\ch(\kk)$. If $A$ is formal in $\ialg_P(\ich(\kk))$ (resp. $N$-formal  in $\ialg_P(\ichp(\kk))$), then $A$ is formal in $\alg_P(\ch(\kk))$ (resp. $N$-formal in $\alg_P(\chp(\kk))$).
\end{prop}

\begin{proof}
Under those assumptions Hinich shows in \cite{Hinichrectif} that the $\infty$-category underlying the model structure on $\alg_P(\ch(\kk))$ for any field is equivalent to the $\infty$-category $\ialg_P(\ich(\kk))$. This immediately yields the result. The statement about $N$-formality follows from  Remark \ref{rem: N formality}. 
\end{proof}

For the rest of this section we fix a positive integer $h$ and a positive rational number $\alpha$ with $\alpha<h$.

\begin{defi}\label{defmpure}Denote by $\ch(gr^{(h)}\Aa)^\apure$ 
the full subcategory of $\ch(gr^{(h)}\Aa)$ given by those $\Z/h$-graded complexes $A=\bigoplus A_n^p$ with \textit{$\alpha$-pure homology}:
\[H_n(A)^p=0\text{ for all }p\not\equiv \alpha n \modh.\]
\end{defi}
 
\begin{prop}\label{formal_functor_abstract_chains}
Let $N=\floor{(h-1)/\alpha}$.
The functor \[U:\chp(gr^{(h)}\Aa)^\apure\lra \chp(\Aa)\] 
defined by forgetting the degree is $N$-formal as a lax symmetric monoidal functor.
\end{prop}
\begin{proof}
We adapt the proof of \cite[Proposition 2.7]{CiHo}, to the $\Z/h$-graded setting. 

Consider the truncation functor $\tau:\chp(gr^{(h)}\Aa)\lra \chp(gr^{(h)}\Aa)$ defined by sending  a $\Z/h$-graded chain complex $A=\bigoplus A_n^p$, to the graded complex given by:
\[(\tau A)_n^p:=\left\{ 
\begin{array}{ll}
A_n^p& n>\ceil{p/\alpha}\\
\Ker(d:A_n^p\to A_{n-1}^p)& n=\ceil{p/\alpha}\\
0& n<\ceil{p/\alpha}
\end{array}\right.\]
for every $n\in\Z_{\geq 0}$ and every $0\leq p<h$. Note that for each $p$, $\tau(A)^p_*$ is the chain complex given by the canonical truncation of $A_*^p$ at $\ceil{p/\alpha}$,
which satisfies 
\[H_n(\tau(A)^p_*)\cong H_n(A^p_*)\text{ for all }n\geq \ceil{p/\alpha}.\]
Using the inequalities of the ceiling function
$\ceil{x}+\ceil{y}-1\leq \ceil{x+y}\leq \ceil{x}+\ceil{y}$ one easily verifies that $\tau$ is lax symmetric monoidal (see the proof of \cite[Proposition 2.7]{CiHo}).

Consider the lax monoidal functor
\[t_{\leq N}H_*:\chp(gr^{(h)}\Aa)\lra \chp(gr^{(h)}\Aa)\] given by the $N$-truncated homology
\[
t_{\leq N}H_n(A)^p:=\left\{ 
\begin{array}{ll}
H_n(A)^p&\text{ if }n\leq N\\
0&\text{ if }n> N\\
\end{array}
\right..
\]

Define a morphism $\Psi(A):\tau A\lra t_{\leq N}H_*(A)$ by letting 
$\Ker(d)\twoheadrightarrow t_{\leq N}H_n(A)^p$ if $n=\ceil{p/\alpha}$ and $A_n^p\to 0$ if $n\neq \ceil{p/\alpha}$.
We next show that this defines a monoidal natural transformation $\Psi:\tau\Rightarrow t_{\leq N}H_*$.
It suffices to check that given $A,B\in \chp(gr^{(h)}\Aa)^\apure$, the diagram
\[\xymatrix{
\ar[d]_{\mu}\tau A\otimes\tau B\ar[rr]^-{\Psi(A)\otimes\Psi(B)}&&t_{\leq N}H_*(A)\otimes t_{\leq N}H_*(B)\ar[d]^{\mu}\\
\tau(A\otimes B)\ar[rr]^-{\Psi(A\otimes B)}&&t_{\leq N}H_*(A\otimes B)
 }
 \]
 commutes. The only non-trivial verification is for elements $a\in (\tau A)_p^n$ and $b\in (\tau B)_{p'}^{n'}$ with
 $n=\ceil{p/\alpha}$ and $n'=\ceil{p'/\alpha}$. 
 Note that we have 
 \[n+n'=\ceil{p/\alpha}+\ceil{p'/\alpha}\geq \ceil{(p+p')/\alpha}.\]
 We have the following three cases:
 
\begin{enumerate}
 \item If $p+p'<h$ and $n+n'=\ceil{(p+p')/\alpha}$ then $\Psi(a\otimes b)=[a\otimes b]\in t_{\leq N}H_{n+n'}(A\otimes B)$ and the diagram commutes.
 \item If $p+p'<h$ and $n+n'>\ceil{(p+p')/\alpha}$ then $\Psi(a\otimes b)=0$. By $\alpha$-purity we have $H_{n+n'}(A\otimes B)^{p+p'}=0$, and the diagram trivially commutes.
 \item If $p+p'\geq h$ then $\Psi(a\otimes b)=0$. Note that we have
 \[n+n'\geq \ceil{(p+p')/\alpha}\geq \frac{p+p'}{\alpha}>\frac{h-1}{\alpha}\geq N.\]
Therefore $t_{\leq N}H_{n+n'}(A\otimes B)=0$ and the diagram trivially commutes.
\end{enumerate}

The inclusion $\tau A\hookrightarrow A$  defines a monoidal natural transformation $\Phi:U\Rightarrow 1$.
Also, there is an obvious monoidal natural transformation $\Upsilon:H_*\Rightarrow t_{\leq N}H_*$ defined by projection.
All together, gives monoidal natural transformations
\[U\stackrel{\Phi}{\Longleftarrow}U\circ \tau \stackrel{\Psi}{\Longrightarrow}U\circ t_{\leq N}H_*=t_{\leq N}H_*\circ U\stackrel{\Upsilon}{\Longleftarrow} H_*\circ U.\]
It only remains to note that, if $A$ has $\alpha$-pure homology, then the morphism $\Phi(A)$ is a quasi-isomorphism and the morphisms
$\Psi(A)$ and $\Upsilon(A)$ are $N$-quasi-isomorphisms.
\end{proof}

\begin{rem}\label{fails_cohomological}
The proof of the above proposition fails in the cohomological case. The main issue is that the functor $\tau$ is not lax symmetric monoidal in the $\Z/h$-graded situation. 
We will provide an alternative statement for cohomological dg-algebras in Section \ref{SecFormaldgas}, using the theory of free models.
\end{rem}

\section{Main results in the homologically graded case}\label{SecFormalChains}
In this section, we use étale chains, together with the formality criterion of Section \ref{SecFormalFunctors},
to prove our main results of formality in their chain version. Let $\alpha$ be a positive rational number with $\alpha<h$, where we recall that $h$ denotes the order of $q$ in $\F_\l^\times$.

Our main result is the following.

\begin{theo}\label{theo functor formality}
Let $\cat{Sch}_K^\apure$ be the category of schemes $X$ over $K$ with the property that $H^n_{et}(X_{\overline{K}},\F_\l)$ is a pure Tate module of weight $\alpha n$, for all $n$. Then the functor \[\inn(\cat{Sch}_K^\apure)\longrightarrow\ich(\F_\l)\] given by  $X\mapsto C_*(X_{an},\F_\l)$ is $N$-formal as a lax symmetric monoidal $\infty$-functor, with $N=\floor{(h-1)/\alpha}$.
\end{theo}

\begin{proof}
We consider the following composition of functors
\[\inn(\cat{Sch}_K^\apure)\xrightarrow{C^*_{et}}\iTComp\op\xrightarrow{\simeq}\ich(\TMod)\op\xrightarrow{G}\ich(gr^{(h)}\Mod_{\F_\l})\op\]
The first functor is the one constructed in \ref{cons: etale cochains}, it does land in $\iTComp$ because of the restriction on the schemes that we consider. The second functor is the equivalence constructed in Theorem \ref{theo : Beilinson for Tate} and the functor $G$ is defined in Lemma \ref{splittingFrobT}. This composition is a symmetric monoidal $\infty$-functor. We can compose this functor with the duality functor $\ich(gr^{(h)}\Mod_{\F_\l})\op\to \ich(gr^{(h)}\Mod_{\F_\l})$ which is also a symmetric monoidal $\infty$-functor when restricted to chain complexes with finite dimensional cohomology (by \cite[Corollary 1.10]{sga4.5} the étale cohomology of a finite type scheme is indeed finite dimensional). We end up with a symmetric monoidal $\infty$-functor
\[\inn(\cat{Sch}_K^\apure)\lra\ichp(gr^{(h)}\Mod_{\F_\l})^\apure.\]
We can further compose this with the forgetful functor
\[\ichp(gr^{(h)}\Mod_{\F_\l})^\apure\lra \ichp({\F_\l})\]
which is $N$-formal by Proposition \ref{formal_functor_abstract_chains}. We end up with an $N$-formal symmetric monoidal $\infty$-functor
\[\inn(\cat{Sch}_K^\apure)\lra\ichp({\F_\l}).\]
This functor is equivalent to the functor $X\mapsto C^*_{et}(X,\F_\l)^{\vee}$ (where $(-)^\vee$ denotes the dual vector space). By the comparison between étale and singular cohomology (Proposition \ref{prop : comparison}), we have an equivalence of symmetric monoidal $\infty$-functors
\[C^*_{et}(X,\F_\l)^{\vee}\simeq C^*(X_{an},\F_\l)^{\vee}.\]
Finally, we have an equivalence of symmetric monoidal $\infty$-functors
\[C_*(X_{an},\F_\l)\simeq C^*(X_{an},\F_\l)^{\vee}\]
since $X_{an}$ has finite dimensional singular cohomology.
\end{proof}

As an immediate corollary of the above theorem and of Proposition \ref{prop: rigidification positive char} we get the following theorem.

\begin{theo}\label{maintorsion}
Let $P$ be an operad in sets and let $X$ be a  $P$-algebra in $\cat{Sch}_K$.  Let $N=\floor{(h-1)/\alpha}$. Assume that for each color $c$ of $P$, the cohomology $H^n_{et}(X(c)_{\overline{K}},\F_\l)$ is a pure Tate module of weight $\alpha n$, for all $n$. Then:
\begin{enumerate}
\item There is an $N$-equivalence $C_*(X_{an},\F_\l)\simeq H_*(X_{an},\F_\l)$ in the $\infty$-category of $P$-algebras in $\ich(\F_\l)$.
\item If $P$ is admissible and $\Sigma$-cofibrant, then $C_*(X_{an},\F_\l)$ is $N$-formal as a dg-$P$-algebra.
\end{enumerate}
\end{theo}

Let us now review some applications of Theorem \ref{maintorsion}.
Consider the cyclic operad $\overline{\mathcal{M}}_{0,\bullet}$.
As mentioned before, this operad lives in the category of smooth and proper schemes over $\Z$. We have:

\begin{theo}\label{dmformal}
The cyclic dg-operad $C_*((\overline{\mathcal{M}}_{0,\bullet})_{an},\F_\l)$ is $2(\l-2)$-formal.
\end{theo}
\begin{proof}
It suffices to prove that for all $n$, the degree $n$ étale cohomology of $\overline{\mathcal{M}}_{0,\bullet}$ with coefficients in $\F_\l$ is a Tate module that is pure of weight $\frac{n}{2}$. We call a scheme $\frac{1}{2}$-pure if it has this property.

We first make the following claims:
\begin{enumerate}
\item [(i)] Schemes that are $\frac{1}{2}$-pure are stable under finite products.
\item [(ii)]If $Z\to X$ is a closed embedding of smooth schemes and $Z$ and $X$ are $\frac{1}{2}$-pure, then the blow-up $B_Z(X)$ is $\frac{1}{2}$-pure.
\end{enumerate}

The first property is an immediate consequence of the K\"unneth formula in \'etale cohomology. The second property follows from the blow-up formula which gives an equivariant isomorphism
\[H_{et}^*(B_Z(X),\mathbb{F}_\l)\cong H_{et}^*(X,\mathbb{F}_\l)\oplus\left(\bigoplus_{i=0}^{c-1} H_{et}^*(Z,\mathbb{F}_\l)[2i]\otimes_{\mathbb{F}_\l}\mathbb{F}_\l(i)\right),\]
where $c$ is the dimension of $Z$.

Now, we can prove that the scheme $\overline{\mathcal{M}}_{0,n}$ is $\frac{1}{2}$-pure by induction on $n$. For $n=3$, this moduli space is a point.
For $n=4$, we have $\overline{\mathcal{M}}_{0,4}\cong\mathbb{P}^1$ and the proposition is a classical computation. Assume that the proposition has been proved for $\{3,4,\ldots, n\}$. We may use Keel's inductive description of $\overline{\mathcal{M}}_{0,n+1}$ as a sequence of blow-ups starting from $\overline{\mathcal{M}}_{0,n}\times\overline{\mathcal{M}}_{0,4}$ and in which at each stage, the variety that is blown-up is isomorphic to $\overline{\mathcal{M}}_{0,p+1}\times \overline{\mathcal{M}}_{0,q+1}$ with $p+q=n$ (see \cite[Section 1]{Keel92}). We conclude by the induction hypothesis and the first claim of the proof.
\end{proof}

We also have a statement for the little disks operad $\oper{D}$. This does not quite fit our theorem since this is not an operad in the category of schemes. Nevertheless, one can construct a model of $C_*(\oper{D},\F_\l)$ equipped with an action of the Grothendieck-Teichmüller group $\widehat{GT}$, which promotes $C_*(\oper{D},\F_\l)$ to an operad in the category of Tate complexes. 

\begin{theo}\label{littledisksformal}
The dg-operad $C_*(\oper{D},\F_\l)$ is $(\l-2)$-formal.
\end{theo}
\begin{proof}
There is an action of $\widehat{GT}$ on an operad $B\widehat{\oper{PAB}}$ constructed by Drinfeld in \cite{drinfeldquasitriangular} (see also \cite[Section 7]{horelprofinite} for more details). This is an operad in the category of simplicial profinite sets. We can apply $S^\bullet(-,\F_\l)$ to this object (this construction is defined in Section \ref{SecEtale}) and we get a cosimplicial cooperad with an action of $\widehat{GT}$ that we can then dualize to get an operad in simplicial $\F_\l$-vector spaces equipped with an action of $\widehat{GT}$. Finally we can apply the Dold-Kan construction to this operad to end up with a dg-operad. We claim that the resulting operad is quasi-isomorphic to $C_*(\oper{D},\F_\l)$. The exact same statement for the framed little disks operad is proved in \cite[Theorem 9.1]{boavidaoperads}.
The group $\widehat{GT}$ comes with a surjective map
\[\chi_\l:\widehat{GT}\to \Z_\l^\times\]
that factors the cyclotomic character of the absolute Galois group of $\Q$ (see Section 3.1 of \cite{schnepsgrothendieck}). Moreover, the action of an element $\varphi\in \widehat{GT}$ on $H_k(\oper{D}(2),\F_\l)$ is trivial in degree zero and given by $\chi_\l(\varphi)\id$ in degree $1$. This can be proved by using the explicit action of $\widehat{GT}$ on $B\widehat{PAB}(2)\cong B\widehat{\Z}$ where $\widehat{\Z}$ denotes the profinite completion of the group $\Z$. As an operad, $H_*(\oper{D},\F_\l)$ is generated by operations of arity $2$. Therefore, we can deduce that the action of $\varphi$ on $H_n(\oper{D}(*),\F_\l)$ is given by multiplication by $\chi_\l(\varphi)^n$ (Petersen uses the analogous argument in the rational case in \cite{petersenminimal}). In particular, let $p$ be a prime number such that $p$ generates $\F_\l^{\times}$ and let $\varphi$ be an element of $\widehat{GT}$ such that $\chi_\l(\varphi)=p$ (such $\varphi$ exists by surjectivity of $\chi_\l$). Then from the previous discussion the pair $(C_*(\oper{D},\F_\l),\varphi)$ is an operad in the category of Tate complexes that is pure of weight $1$. 
\end{proof}

Let us denote by $\oper{D}_{\leq n}$ the truncation of the little disks operad in arity less than or equal to $n$ (that is we only keep the composition maps that only involve those arities). Then $C_*(\oper{D}_{\leq n},\F_\l)$ is an $n$-truncated dg-operad.

\begin{coro}\label{trunclittledisksformal}
The $(\l-1)$-truncated operad $C_*(\oper{D}_{\leq (\l-1)},\F_\l)$ is formal.
\end{coro}

\begin{proof}
By the previous theorem, we deduce that $C_*(\oper{D}_{\leq (\l-1)},\F_\l)$ is $(\l-2)$ formal. On the other hand the homology of $\oper{D}(n)$ is concentrated in degrees less than or equal to $n-1$. Therefore, the $\l-2$-truncation map 
\[C_*(\oper{D}_{\leq \l-1},\F_\l)\to t_{\leq \l-2}C_*(\oper{D}_{\leq \l-1},\F_\l)\]
is a quasi-isomorphism.
\end{proof}

\begin{rem}\label{remarkoptimal}
This corollary is in some sense optimal. Indeed $C_*(\oper{D}_{\leq \l},\F_\l)$ is not formal. If it were the case, it would imply that there is a $\Sigma_\l$-equivariant quasi-isomorphism.
\[C_*(\oper{D}(\l),\F_\l)\simeq H_*(\oper{D}(\l),\F_\l)\]
This would mean that the homotopy orbit spectral sequence
\[H_*(\Sigma_\l,H_*(\oper{D}(\l),\F_\l))\implies H_*(\oper{D}(\l)_{h\Sigma_\l},\F_\l)\]
collapses at the $E^2$-page. However, this is not the case. Indeed the $E^2$-page has non-trivial homology in arbitrarily high degree. On the other hand, the space $\oper{D}(\l)_{h\Sigma_\l}$ has the homotopy type of the space of unordered configurations of $\l$ points in the plane which is a manifold and in particular has bounded above homology.
\end{rem}

\begin{rem}
Formality of the operad of little disks with torsion coefficients is generalized to higher dimensional little disks operads in
the work \cite{BH} of the second author with Pedro Boavida de Brito. The strategy of proof
is exactly the same: we construct a Frobenius automorphism on the singular chains over the
little $n$-disks operad in such a way that the resulting operad in Tate complexes is pure.
\end{rem}

We have a similar theorem for the framed little disks operad $\oper{FD}$.

\begin{theo}\label{framedlittledisksformal}
The dg-operad $C_*(\oper{FD},\F_\l)$ is $(\l-2)$-formal. 
\end{theo}

\begin{proof}
The proof is entirely analogous once we have an action of $\widehat{GT}$ on the framed little disks operad. This is constructed in \cite[Theorem 8.4]{boavidaoperads} and the effect of this action on homology is explained in \cite[Theorem 9.1]{boavidaoperads}.
\end{proof}

\section{Weight decompositions and formality of cohomological dg-algebras}\label{SecFormaldgas}
We next give a criterion of formality for dg-algebras over an arbitrary field $\kk$ equipped with a $\Z/h\Z$-graded weight decomposition. In this section and the following we use cohomological instead of homological grading. 

Fix a positive integer $h$ and a positive rational number $\alpha$ with $\alpha<h$.

\begin{defi}
Let $A$ be a non-negatively graded dg-algebra over $\kk$.
A \textit{$gr^{(h)}$-weight decomposition} of $A$
is a  direct sum decomposition
\[A^n=\bigoplus_{p=0}^{h-1} A_p^n\]
of each vector space $A^n$, such that:
\begin{enumerate}
 \item $dA^n_p\subseteq A^{n+1}_p$ for all $n\geq 0$ and all $0\leq p\leq h-1$.
 \item $A_p^n\cdot A_{p'}^{n'}\lra A_{p+p' \modh}^{n+n'}$ for all $n,n'\geq 0$ and all $0\leq p,p'\leq h-1$.
\end{enumerate}
Given $x\in A^n_p$ we will denote by $|x|=n$ its \textit{degree} and by $w(x)=p$ its \textit{weight}.
\end{defi}

Denote by $\dga{h}{\kk}$ the category of dg-algebras with $gr^{(h)}$-weight decompositions.
Note that these are just monoids in $\cochp(gr^{(h)}\kk)$.

\begin{defi}
The cohomology of a dg-algebra $A\in \dga{h}{\kk}$ is said to be \textit{$\alpha$-pure}
if
\[H^n(A)_p=0\text{ for all }p\not\equiv \alpha n \modh.\]
\end{defi}
Denote by $\dga{h}{\kk}^{\apure}$ the category of dg-algebras with $\alpha$-pure cohomology.

\subsection{Massey products}
We first prove some general vanishing results of Massey products in cohomology.
Let $A$ be a dg-algebra and $x_1,\cdots,x_k\in H^*(A)$ cohomology classes, with $k\geq 3$. A \textit{defining system} for $\{x_1,x_2,\cdots,x_k\}$ is a collection of elements $\{x_{i,j}\}$, for $1\leq i\leq j\leq k$ with $(i,j)\neq (1,k)$
where $x_i=[x_{i,i}]$ and 
\[d(x_{i,j})=\sum_{q=i}^{j-1} (-1)^{|x_{i,q}|} x_{i,q} x_{q+1,j}.\]
Consider the cocycle 
\[\gamma(x_{i,j}):=\sum_{q=1}^{k-1}(-1)^{|x_{1,q}|} x_{1,q} x_{q+1,k}.\]
The \textit{$k$-tuple Massey product} $\langle x_1,\cdots,x_k\rangle$ is defined to be the set of all cohomology classes
$[\gamma(x_{i,j})]$, for all possible defining systems. 
A Massey product is said to be \textit{trivial} if the trivial cohomology class belongs to its defining set.

\begin{rem}
Note that the triple Massey product $\langle x_1,x_2,x_3\rangle$ is empty unless
$x_1x_2=0$ and $x_2x_3=0$. For $k>3$ one similarly asks that some $q$-tuple Massey products, with $q<k$,
are trivial in a certain compatible way, so that at least one defining system exists.
\end{rem}

\begin{rem}Given any dg-algebra defined over a field, there is a transferred
structure of $A_\infty$-algebra on its cohomology $H^*(A)$, 
which is unique up to $A_\infty$-isomorphism. The higher operations $\mu_k$ of this $A_\infty$-structure give elements in the corresponding Massey sets. Conversely, 
given $x\in \langle x_1,\cdots,x_k\rangle$ there is always an $A_\infty$-structure on $H^*(A)$ such that $\mu_k(x_1,\cdots,x_k)=\pm x$ \cite{BMM}.
\end{rem}

By obvious degree-weight reasons, the condition of $\alpha$-purity has direct
consequences on the vanishing of Massey products. We have:

\begin{prop}\label{prop: vanishing of Massey products}
Let $A\in \dga{h}{\kk}^{\apure}$ and $k\geq 3$ an integer. If $\frac{\alpha (k-2)}{h}\notin\mathbb{Z}$, then all $k$-tuple Massey products in $H^*(A)$ are trivial.
\end{prop}
\begin{proof}
Assume that $x_i\in H^{n_i}(A)_{w_i}$, for $i=1,\cdots,k$, so that $w(x_i)=w_i$ and $|x_i|=n_i$. We will show that $\langle x_1,\cdots,x_k\rangle$ is trivial. Given a defining system $\{x_{i,j}\}$ we have  
\[|x_{ij}|=\sum_{q=i}^j n_i +i-j\text{ and }w(x_{ij})=\sum_{q=i}^j w_i.\]
Let $n:=\sum_{i=1}^k n_i$ and $w:=\sum_{i=1}^k w_i$.
The above expression for $\gamma(x_{i,j})$ gives
\[|\gamma(x_{i,j})|=n-k+2\text{ and } w(\gamma(x_{i,j}))=w.\]
Therefore the $k$-tuple Massey product $\langle x_1,\cdots,x_k\rangle$ lives in
$H^{n-k+2}(A)_{w}$.
Now, $\alpha$-purity tells us that $w_i\equiv \alpha n_i \modh$ and hence
$H^{n-k+2}(A)_{w}$ is non-trivial only when $w\equiv \alpha (n-k+2) \modh$.
This gives the condition $\alpha (k-2) \equiv 0 \modh$.
\end{proof}

For simply connected dg-algebras,
the above proposition tells us that all higher Massey products 
living in sufficiently low-degree cohomology will be trivial, regardless of their length.

\begin{defi}
A dg-algebra $A$ is said to be \textit{cohomologically connected} 
if the unit map induces an isomorphism $H^0(A)\cong \kk$. It is
\textit{simply connected} if, in addition, $H^1(A)=0$. Let $r\geq 1$. Then $A$ is called \textit{$r$-connected} if, in addition, $H^i(A)=0$ for all $1\leq i\leq r$.
\end{defi}

\begin{prop}\label{NMassey}
Let $A\in \dga{h}{\kk}^{\apure}$ be simply connected. Then
$H^{\leq N}(A)$ has no non-trivial Massey products, where $N=\ceil{\frac{h}{\alpha}}+3$.
More generally, if $A$ is $r$-connected with $r>0$ then 
$H^{\leq N_r}(A)$ has no non-trivial Massey products, where
$N_r=\ceil{\frac{hr}{\alpha}}+2r+1$.
\end{prop}
\begin{proof}
If $H^i(A)=0$ for all $0<i\leq r$, then a $k$-tuple Massey product will have degree at least $rk+2$.
The condition $rk+2\leq \ceil{\frac{hr}{\alpha}}+2r+1$ gives $\alpha(k-2)<h$.
\end{proof}

\subsection{Formality of dg-algebras}
We now prove partial formality results for dg-algebras with pure weight decompositions in cohomology.

\begin{defi}Let $N\geq 0$ be an integer.
A dg-algebra $A$ is said to be \textit{$N$-formal} if there is a string of $N$-quasi-isomorphisms
of dg-algebras from $A$ to its cohomology $H^*(A)$, considered as a dg-algebra with trivial differential. 
\end{defi}

As mentioned in Remark \ref{fails_cohomological}, a main issue in the setting of $\Z/h$-graded dg-algebras is that the functor $\tau$ used in the proof of 
Proposition \ref{formal_functor_abstract_chains} is not lax monoidal. 
To circumvent this, given a 
$\Z/h$-graded dg-algebra with $\alpha$-pure cohomology we will take a free model which is actually $\Z$-graded.
With this new weight-grading we will get $\alpha$-purity only up to degree $N\leq (h-1)/\alpha$, which leads to formality up to this degree.

\begin{prop}\label{Nformaldg-algebras}
Every simply connected dg-algebra in $\dga{h}{\kk}^\apure$ is $N$-formal, with $N=\lfloor\frac{(h-1)}{\alpha}\rfloor$.
\end{prop}

\begin{proof}
Our first task is to construct a convenient model for $A$. We claim that there exists an $N$-quasi-isomorphism $f:M\to A$ in the category $\dga{h}{\kk}^\apure$ such that:
\begin{enumerate}
 \item $M$ is generated by elements of degree $\leq N$.
 \item The differential on $M^n_{\alpha n}$ is trivial for all $\alpha n\leq h-1$. 
 \item Let $\alpha n<p\leq h-1$. If $M_p^n\neq 0$ then $p\leq \alpha(2n-2)$.  
\end{enumerate}

We use the standard theory of free models for dg-algebras over an arbitrary field (see \cite{HaLe}) in order to construct such an $M$. We will define, inductively over the degree $n\geq 0$,
a sequence of dg-algebras $M\langle n\rangle$ with $gr^{(h)}$-weight decompositions 
$M\langle n\rangle^k=\bigoplus M\langle n\rangle^k_p$
and morphisms $f_n:M\langle n\rangle\to A$ compatible with weight decompositions and satisfying the following conditions:
\begin{enumerate}
 \item [$(a_n)$] The dg-algebra $M\langle n\rangle$ is a free extension of $M\langle n-1\rangle$ by generators of degree $n$ 
 and weights $w=p \modh$, with  $\alpha n\leq p\leq \alpha(2n-2)$.
 \item [$(b_n)$] The map $H^i(f_n)$ is an isomorphism for all $i\leq n$ and a monomorphism for $i=n+1$.
\end{enumerate}
Then the morphism \[f:\bigcup_n f_n:\bigcup_{n\leq N} M\langle n\rangle \to A\] will be our model for $A$.

Let $M\langle 1\rangle=\kk$ concentrated in weight 0 and degree 0 and define $f_1:M\langle 1\rangle\to A$ to be the unit map.
Conditions $(a_1)$ and $(b_1)$ are trivially satisfied.
Assume inductively that we have defined $f_{n-1}:M\langle n-1\rangle\to A$ satisfying $(a_{n-1})$ and $(b_{n-1})$.
for each $0\leq p<h$, let $V_p:=H^n(C(f_{n-1}))_p$ and consider it as a bigraded vector space of degree $n$ and pure weight $p$.
Define a differential $d:V_p\to M\langle n-1\rangle^{n+1}_p$ and a map $f_n:V_p\to A^n_p$
by taking a section of the projection
\[H^n(C(f_{n-1}))_p\twoheadleftarrow Z^n(C(f_{n-1}))_p\subset M\langle n-1\rangle^{n+1}_p\oplus A^n_p.\]
These define a differential 
on \[M\langle n\rangle :=M\langle n-1\rangle \sqcup T(V)\]
and a map $f_n:M\langle n\rangle\to A$ compatible with the weight decompositions.
By classical arguments, since $A$ is simply connected and $H^n(A)=\bigoplus H^n(A)_p$, condition $(b_n)$ is satisfied.
Let us prove $(a_n)$. Note that elements in $V_p$ arise either from the cokernel of $H^n(f)_p$, which by $\alpha$-purity, 
is non-trivial only for $p=\alpha n$,
or from elements in the kernel of $H^{n+1}(f)_p$.
Since $M^1=0$, elements in this kernel are represented by sums of products $x_1\cdots x_k\in M\langle n-1\rangle^{n+1}$ with $k\geq 2$.
Let $n_i:=|x_i|$. By induction hypothesis we have
$w(x_i)\equiv p_i  \modh$ with $\alpha n_i\leq p_i\leq \alpha(2n_i-2)$.
We get
\[\alpha(n+1)=\alpha(\sum_{i=1}^k n_i)\leq \sum_{i=1}^k p_i \leq  
\alpha\sum_{i=1}^k(2n_i-2)\leq 2\alpha (n+1) -2\alpha k\leq \alpha(2n-2).\]
Therefore the generators of $V_p$ have degree $n$ and weights 
$w\equiv p \modh$ with 
$\alpha n\leq p\leq (2n-2)\alpha$ and condition $(a_n)$ is satisfied.
This ends the inductive step.

We now prove that the differential on $M^n_{\alpha n}$ is trivial for all $\alpha n\leq h-1$. In fact, we will show that 
$M_{\alpha n}^{n+1}=0$ for all $\alpha n\leq h-1$.
Assume that $x\in M_{\alpha n}^{n+1}$. By construction, $w(x)=\alpha n=p-\lambda n$ with $\lambda\in\Z_{> 0}$
and $\alpha(n+1)\leq p\leq 2n\alpha$. This gives the condition $\lambda h\leq \alpha n$, which contradicts the condition that 
$\alpha n\leq h-1$. Therefore $x=0$.

\medskip

Now, we are ready to prove the proposition. It suffices to prove that the model $M$ that we have just constructed is $N$-formal. As a bigraded vector space, $M$ may be decomposed into $M=A\oplus D\oplus B$ where $D$ denotes the diagonal $p=\alpha n$ truncated up to degree $\alpha n<h$, and $A$ and $B$ are the direct sum of all vector spaces above and below the diagonal respectively:
by letting
\[I_d:=\{(n,p); p=\alpha n\leq h-1\}\text{ and } I_a:= \{(n,p); n< \frac{h-1}{\alpha}, p>\alpha n  \}\]
we may write
\[D=\bigoplus_{(n,p)\in I_d}M^n_p, 
 A=\bigoplus_{(n,p)\in I_a}M^n_p\text{ and }
 B=\bigoplus_{(n,p)\notin I_d\cup I_a}M^n_p. 
\]

By degree-weight reasons, $B$ is a dg-algebra ideal of $M$. By construction of $M$, the
differential of $D$ is trivial and
we have  $H^n(B)=0$ for all $n\leq (h-1)/\alpha$.
Therefore the morphism of dg-algebras $\pi:M\twoheadrightarrow M':=M/B$ induces an isomorphism in cohomology $H^n(\pi)$ for all $n\leq (h-1)/\alpha$.
Consider the projection morphism $M'\to H^*(M')$
given by $A\mapsto 0$ and $D\mapsto D/\mathrm{Im}(d)$.
To see that it is a quasi-isomorphism of dg-algebras, it suffices to see that 
$A\times M'\subseteq A$.
Let $x\in A$ and $y\in M'$. By construction of $M$ we have 
\[
\begin{array}{lll} 
 w(x)\equiv p_x \modh &\text{ with }&\alpha |x|<p_x\leq (2|x|-2)\alpha, \\
 w(y)\equiv p_y \modh & \text{ with }& \alpha |y|\leq p_y\leq (2|y|-2)\alpha.
 \end{array}
\]
Assume that $0\neq x\cdot y\in D$ and let $n:=|x|+|y|$. Then there is $\lambda\in \Z$ such that
$p_x+p_y=\alpha n+\lambda h$. Since $\alpha n<p_x+p_y$ we have $\lambda_\in\Z_{>0}$.
This gives
$\alpha n+ h\leq \alpha n+\lambda h\leq 2\alpha n-4\alpha$ and hence $h\leq \alpha(n-4)$ which is a contradiction, since $\alpha n\leq h-1$.

\end{proof}

\section{Main results in the cohomologically graded case}\label{SecFormalCochains}

Recall that objects in the category of Tate complexes $\TComp$ are chain complexes over $\F_\l$ enriched with automorphisms $\varphi$ giving Tate modules in cohomology.
We will show that every monoid in 
$\TComp$  is quasi-isomorphic to a dg-algebra in $\dga{h}{\F_\l}$, where 
$h$ denotes the order of $q$ in  $\F_\l^\times$. 
As a consequence, the formality criterion of Proposition \ref{Nformaldg-algebras} applies to étale cochains.

\begin{prop}\label{tatetogr}
Let $A$ be a cohomologically connected monoid in $\TComp$.
Then there is a dg-algebra $M$ in $\dga{h}{\F_\l}$ together with a quasi-isomorphism $M\to A$ of dg-algebras,
such that $H^n(M)_k$ corresponds to the generalized eigenspace of $(H^n(A),\varphi)$ 
of eigenvalue $q^k$. 
\end{prop}

\begin{proof}
Using Theorem \ref{theo : Beilinson for Tate}, we can simply take $M$ to be the subalgebra $\iota\pi(A)$
\end{proof}

Let $\alpha$ be a positive rational number, with $\alpha<h$.
Our main theorem for cohomological dg-algebras is the following:
\begin{theo}\label{maindg-algebras}
Let $X\in \cat{Sch}_K$ be a scheme over $K$.
Assume that for all $n$, $H^n_{et}(X_{\overline{K}},\F_\l)$ is a pure Tate module of weight $\alpha n$.
Then the following is satisfied:
\begin{enumerate}
 \item [(i)]  If $\alpha(k-2)/h\notin\mathbb{Z}$ then all $k$-tuple Massey products are trivial in $H^*(X_{an},\F_\l)$.
 \item [(ii)]If $H^i(X_{an},\F_\l)=0$ for all $0<i\leq r$ then $H^n(X_{an},\F_\l)$ contains no non-trivial Massey products
 for all $n\leq \ceil{\frac{hr}{\alpha}+2r+1}$.
  \item [(iii)] If it is simply connected, then the dg-algebra
  $C^*_{sing}(X_{an},\F_\l)$ is $N$-formal, with $N=\floor{(h-1)/\alpha}$.
\end{enumerate}
\end{theo}

\begin{proof}
By assumption, the dg-algebra $C^*_{sing}(X_{an},\F_\l)$ is quasi-isomorphic to a monoid in Tate complexes
$(C^*_{et}(X,\F_\l),\varphi)$ 
with $\alpha$-pure cohomology.
Therefore by Proposition \ref{tatetogr}, it is quasi-isomorphic to a dg-algebra in $\dga{h}{\F_\l}^\apure$.
It now suffices to apply Propositions \ref{prop: vanishing of Massey products}, \ref{NMassey}, \ref{Nformaldg-algebras},
respectively for each of the implications (i), (ii) and (iii).
\end{proof}

We have the following direct application of the previous theorem.

\begin{coro}
Let $X$ be a smooth and proper scheme over $\mathcal{O}_K$, the ring of integers of $K$. Assume that the only eigenvalues of the Frobenius action on $H^*_{et}(X_{\overline{K}},\F_\l)$ are powers of $q$. Then:
\begin{enumerate}
 \item [(i)] If $(k-2)/{2h}\notin\mathbb{Z}$ then all $k$-tuple Massey products are trivial in $H^*(X_{an},\F_\l)$.
 \item [(ii)]If $H^i(X_{an},\F_\l)=0$ for all $0<i\leq r$ then $H^n(X_{an},\F_\l)$ contains no non-trivial Massey products
 for all $n\leq 2hr+2r+1$.
  \item [(iii)] If it is simply connected, the dg-algebra $C^*_{sing}(X_{an},\F_\l)$ is $2(h-1)$-formal.
\end{enumerate}
\end{coro}

\begin{proof}
The conditions of the theorem are satisfied with $\alpha=1/2$.
\end{proof}

\begin{example}
We can apply this corollary to $\mathbb{P}^n$. This is a smooth and proper scheme over $\Z$. We can therefore base change it to $\Z_p$ with $p$ a prime number that generates $\F_\l^\times$ and we deduce that $C^*_{sing}(\mathbb{P}^n_{an},\F_\l)$ is $2(\l-2)$-formal. Note that in fact, complex projective spaces are formal over the integers (see \cite{Berg}).

As illustrated by this example, it often happens that the scheme of interest has a smooth and proper model $X$ over a commutative ring $R$ that is finitely generated over $\Z$. In that case there are infinitely many ways to base change $X$ to a ring of the form $\mathcal{O}_K$ (one for each maximal ideal of $R$). Those different ways give rise to different values for the parameter $h$ and one should pick the option that yields the largest possible value for $h$.
\end{example}

Denote by $F_m(\mathbb{A}^d)$ the scheme of configurations of $m$ points in $\mathbb{A}^d$.
This is informally described at the point set level as follows
\[F_m(\mathbb{A}^d)=\{(x_1,\ldots,x_m)\in(\mathbb{A}^d)^m| x_i\neq x_j,\;1\leq i<j\leq m\}\]

The configuration spaces $F_m(\R^d)$ are known to be formal over $\Q$ for any $m$ and $d$.
However, the question of formality over $\F_\l$ is open as explained in \cite{Salvatore}.
Using our machinery we deduce some results of formality over $\F_\l$ for configuration spaces $F_m(\mathbb{C}^d)$.
We will in fact consider the more general problem of the formality of a complement of subspace arrangements. 

\begin{defi}\label{defi : good arrangement}
Let $V$ be a $d$-dimensional $K$-vector space. We say that a finite family $\{W_i\}_{i\in I}$ subspaces of $V$ is a \textit{good arrangement of codimension $c$ subspaces} if each $W_i$ is of codimension $c$ in $V$ and for each subset $J\subset I$, the codimension of the intersection $\bigcap_{j\in J} W_j$ is a multiple of $c$.
\end{defi}

\begin{rem}
It is easy to check that this definition implies the following properties.
\begin{enumerate}
\item The empty arrangement is good.
\item If $\{W_i\}_{i\in I}$ is a good arrangement of codimension $c$ subspaces of $V$, then, for each $J\subset I$, the family $\{W_j\}_{j\in J}$ is a good arrangement of codimension $c$ subspaces of $V$.
\item If $\{W_i\}_{i\in I}$ is a good arrangement of codimension $c$ subspaces of $V$, then for each $i\in I$, the family $\{W_j\cap W_i\}_{j\neq i}$ is either a good arrangement of codimension $c$ subspaces of $W_i$ or one of the spaces $W_j\cap W_i$ is equal to $W_i$.
\end{enumerate}
\end{rem}

\begin{rem}\label{rem : good arrangements}
In \cite[Definition 8.2]{CiHo} we gave an inductive definition of good arrangements which was slightly imprecise. We believe that Definition \ref{defi : good arrangement} is more natural and clearer. Under this definition we redo the proof of \cite[Proposition 8.6]{CiHo} in the context of Galois actions in Lemma \ref{purity_codim} below.
\end{rem}

\begin{example}
Any hyperplane arrangement is a good arrangement of codimension $1$-subspaces.
\end{example}

\begin{example}
Recall that a set of subspaces of codimension $c$ in a $d$-dimensional vector space is said to be in general position if the intersection of $n$ of those subspaces is of codimension $\on{min}(d,cn)$. One easily checks that, if $d$ is a multiple of $c$, a set of codimensions $c$ subspaces in general position is a good arrangement. Let us mention that there is small mistake in \cite[Example 8.4]{CiHo} where we forgot to include the condition that the dimension of the ambient space is a multiple of $c$. Without this hypothesis \cite[Proposition 8.6]{CiHo} does not hold as can be seen by considering the complement of $2$ distinct lines in $\mathbb{C}^3$.
\end{example}

\begin{example}
Take $V=(K^d)^m$ and define, for $(i,j)$ an unordered pair of distinct elements in $\{1,\ldots,m\}$, the subspace
\[W_{(i,j)}=\{(x_1,\ldots,x_m)\in (K^d)^m, x_i=x_j\}.\]
This collection of codimension $d$ subspaces of $V$ is a good arrangement. The complement $V-\bigcup_{(i,j)}W_{(i,j)}$ is exactly $F_m(\mathbb{A}^d)$, the ordered configuration space of $m$ distinct points in $\mathbb{A}^d$. Let us observe that these subspaces are not in general position if $m$ is at least $3$. Indeed, the codimension of $W_{(1,2)}\cap W_{(1,3)}\cap W_{(2,3)}$ is $2d$. This example shows that good arrangements are more general than arrangements in general position.
\end{example}

\begin{lemm}[c.f. \cite{BjEk}]
\label{purity_codim}
Let $V$ be a $d$-dimensional vector space over $K$ and $\{W_i\}_{i\in I}$ be a finite collection of subspaces that form a good arrangement of codimension $c$ subspaces. Then $H^n_{et}(V-\bigcup_i W_i,\F_\l)$ is pure of weight $cn/(2c-1)$.
\end{lemm}

\begin{proof}
We proceed by induction on the cardinality of $I$. When $I$ is empty there is nothing to prove. Assume that the lemma has been proven for $|I|<n$. Let $j$ be an element of $I$. Let $X=V-\bigcup_{i\neq j} W_i$, let $U=V-\bigcup_i W_i$, then $U$ is an open subscheme of $X$ whose complement is $Z=W_j-\bigcup_{i\neq j} W_i\cap W_j$. By remark \ref{rem : good arrangements}, the scheme $X$ is a complement of good arrangements of $|I|-1$ codimension $c$ subspaces and the scheme $Z$ is either empty or a complement of good arrangements of $|I|-1$ codimension $c$ subspaces. We have a Gysin long exact sequence
\[\ldots\to H^{n-2c}_{et}(Z,\F_\l)(-c)\to H^n_{et}(X,\F_\l)\to H^n_{et}(U,\F_\l)\to H^{n+1-2c}_{et}(Z,\F_\l)(-c)\to \ldots\]
By the induction hypothesis, both $H^n_{et}(X,\F_\l)$ and $H^{n+1-2c}_{et}(Z,\F_\l)(-c)$ are of weight $cn/(2c-1)$, thus $H^n_{et}(U,\F_\l)$ is also of weight $2cn/(2c-1)$ as desired.
\end{proof}

\begin{theo}\label{formal_subspaces} Let $X$ be a complement of a good arrangement of codimension $c$ subspaces defined over $K$. Then:
\begin{enumerate}
 \item [(i)]If $(2c-1)(k-2)/hc\notin\mathbb{Z}$ then all $k$-tuple Massey products are trivial in $H^*(X_{an},\F_\l)$.
 \item [(ii)]The space $H^n(X_{an},\F_\l)$ contains no non-trivial Massey products for 
 \[n\leq \left\ceil{\frac{h (2c-1)(2c-2)}{c}+4c-3\right}\]
  \item [(iii)] If it is simply connected, the dg-algebra $C^*_{sing}(X_{an},\F_\l)$ is $N$-formal, with \[N=\floor{(h-1)(2c-1)/c}.\]
\end{enumerate}
\end{theo}

\begin{proof}By Lemma \ref{purity_codim} we know that codimension $c$ subspace arrangements satisfy the
conditions of Theorem \ref{maindg-algebras}. For part (ii), one shows by induction on the number of subspaces that the cohomology of $X$ vanishes in degree $\leq 2c-2$.
\end{proof}

\begin{rem} \label{rem: Matei}
By (i) of the above Theorem we have that for complements of hyperplane arrangements defined over $K$,
all triple Massey products over $\F_\l$ are trivial, as long as $h>1$.
In \cite{Matei}, Matei showed that for every odd prime $\l$, there is a
(non-simply connected)
complement of hyperplane arrangements $X$ in $\C^3$ with a non-trivial triple Massey product in $H^2(X,\F_\l)$.
These two facts do not contradict each other since Matei's hyperplane arrangement cannot be modeled over a $p$-adic field $K$ with residue field $\F_q$ unless $\l$ divides $q-1$ (indeed, it requires $K$ to have all $\l$-th root of unity) but in that case $h=1$ and Theorem \ref{formal_subspaces} is vacuous.
\end{rem}

The above result applies to configuration spaces of $m$ points in $\mathbb{C}^d$.  We get:
\begin{theo}\label{formal_config}
Let $d\geq 2$. For any finite field $\F_\l$, the dg-algebra 
$C_{sing}^*(F_m(\mathbb{C}^d),\F_\l)$ is $N$-formal, with
\[N=\left\floor{\frac{(\l-2)(2d-1)}{d}\right}.\]
\end{theo}

\begin{proof}
The space $F_m(\mathbb{C}^d)$ is the complement of a good codimension $d$ subspace arrangement defined over any of the fields $\Q_p$ (in fact it can be defined over $\Z$). We can thus pick a prime number $p$ such that $p$ generates $\F_\l^\times$ and we get the desired result from the previous theorem.
\end{proof}

\begin{coro}\label{coroconfig}
Let $d\geq 2$. For any finite field $\F_\l$, the dg-algebra $C_{sing}^*(F_m(\mathbb{C}^d),\F_\l)$ is formal when $\l\geq (m-1)d+2$.
\end{coro}

\begin{proof}
The cohomology of $C_{sing}^*(F_m(\mathbb{C}^d),\F_\l)$ is concentrated in degree $\leq (m-1)(2d-1)$. Therefore, by the previous theorem, this dg-algebra is formal when
\[(m-1)(2d-1)\leq \frac{(\l-2)(2d-1)}{d}.\qedhere\]
\end{proof}

\begin{example}
Consider the configuration space $F_m(\C)$ of $m$ points in $\C$. In \cite{Salvatore} it is shown that $C^*_{sing}(F_m(\C),\F_2)$ is not formal for any $m\geq 4$, and the question of whether 
$C^*_{sing}(F_m(\C),\F_\l)$ is formal for $\l>2$ is left open. 
Theorem \ref{formal_subspaces} ensures that if
$\l\geq k$ then all $k$-tuple Massey products are trivial in $H^*(F_m(\C),\F_\l)$.
This partially answers a question raised at the end of \cite{Salvatore},
asking how far one has to go on the filtered model to find obstructions to formality of the dg-algebra $C^*_{sing}(F_m(\C),\F_\l)$. 
\end{example}

\begin{rem}
In \cite{BH}, similar formality results are obtained for the spaces $F_m(\R^d)$ with $d$ not necessarily even. The strategy of proof is the same as the one used here, after building an automorphism on the dg-algebra of singular cochains that plays the role of the Frobenius automorphism. However, contrary to the situation here, that automorphism does not come directly from algebraic geometry.

In \cite{drummondhomotopy}, these formality results are proved over the ring $\Z_\l$ instead of $\F_\l$. This paper uses homotopy transfer techniques instead of the $\infty$-categorical methods used in the present paper.
\end{rem}

\appendix

\section{A counter-example}

In this appendix, we produce an example showing that Proposition \ref{Nformaldg-algebras} is incorrect if we remove the simple connectivity assumption. Consider the following commutative dg-algebra over any field $\kk$:
\[B=\Lambda(x,y,z_0,z_1,z_2), dx=dy=0, dz_0=xz_2, dz_1=xz_0, dz_2=xy\]
Let $I$ be the dg-ideal of $B$ generated by $yz_2,yz_0$ and $xz_0z_1$ and let $A=B/I$.

A straightforward computation shows that the non-zero cohomology groups of $A$ are as follows
\begin{align*}
H^0(A)&=\kk[1]\\
H^1(A)&=\kk[x]\oplus\kk[y]\\
H^2(A)&=\kk[xz_1]\oplus\kk[yz_1]\oplus\kk[z_0z_2]\\
H^3(A)&=\kk[xyz_1]\oplus\kk[z_2z_0z_1]
\end{align*}
(where we denote by $[a]$ the cohomology class of a cocycle $a$).

If we assign weight $0$ to $z_0$, weight 1 to $x,y,z_1$ and weight $2$ to $z_2$, the algebra $B$ becomes an object of $gr^{(3)}\cat{DGA}_{\kk}$. Since the ideal $I$ is generated by homogeneous elements, the algebra $A$ is also in $gr^{(3)}\cat{DGA}_{\kk}$, and by inspecting its cohomology we see that it is in $gr^{(3)}\cat{DGA}_{\kk}^{1\text{-}pure}$.

On the other hand, we claim that $A$ is not $2$-formal. This can be seen by observing that there is a non-trivial $5$-tuple Massey product of elements in $H^1(A)$, given by
\[\langle [x],[x],[x],[x],[y]\rangle=\{[xz_1]\}.\]

\bibliographystyle{alpha}

\bibliography{biblio}

\begin{thebibliography}{BdBHR19}

\bibitem[AM69]{ArtinMazur}
M.~Artin and B.~Mazur.
\newblock {\em Etale homotopy}, volume 100 of {\em Lect. Notes Math.}
\newblock Springer, Cham, 1969.

\bibitem[BB20]{Berg}
A.~Berglund and K.~B\"{o}rjeson.
\newblock Koszul {$A_\infty$}-algebras and free loop space homology.
\newblock {\em Proc. Edinb. Math. Soc. (2)}, 63(1):37--65, 2020.

\bibitem[BdBH21]{BH}
P.~Boavida~de Brito and G.~Horel.
\newblock On the formality of the little disks operad in positive
  characteristic.
\newblock {\em J. Lond. Math. Soc. (2)}, 104(2):634--667, 2021.

\bibitem[BdBHR19]{boavidaoperads}
P.~Boavida~de Brito, G.~Horel, and M.~Robertson.
\newblock Operads of genus zero curves and the {G}rothendieck-{T}eichm{ü}ller
  group.
\newblock {\em Geom. Topol.}, 23(1):299--346, 2019.

\bibitem[BE97]{BjEk}
A.~Bj\"orner and T.~Ekedahl.
\newblock Subspace arrangements over finite fields: cohomological and
  enumerative aspects.
\newblock {\em Adv. Math.}, 129(2):159--187, 1997.

\bibitem[BF04]{bergercombinatorial}
C.~Berger and B.~Fresse.
\newblock Combinatorial operad actions on cochains.
\newblock {\em Math. Proc. Cambridge Philos. Soc.}, 137(1):135--174, 2004.

\bibitem[BHH17]{barneapro}
I.~Barnea, Y.~Harpaz, and G.~Horel.
\newblock Pro-categories in homotopy theory.
\newblock {\em Algebr. Geom. Topol}, 17(1):567--643, 2017.

\bibitem[BMFM20]{BMM}
U.~Buijs, J.~M. Moreno-Fern\'{a}ndez, and A.~Murillo.
\newblock {$A_\infty$} structures and {M}assey products.
\newblock {\em Mediterr. J. Math.}, 17(1):Paper No. 31, 15, 2020.

\bibitem[CC21]{ChCiSheaves}
D.~Chataur and J.~Cirici.
\newblock Sheaves of {$E$}-infinity algebras and applications to algebraic
  varieties and singular spaces.
\newblock {\em Trans. Amer. Math. Soc.}, 375(02):925--960, 2021.

\bibitem[CH20]{CiHo}
J.~Cirici and G.~Horel.
\newblock Mixed {H}odge structures and formality of symmetric monoidal
  functors.
\newblock {\em Ann. Sci. \'{E}c. Norm. Sup\'{e}r.}, 53(4):1071--1104, 2020.

\bibitem[CH22]{ciricietale}
J.~Cirici and G.~Horel.
\newblock {\'E}tale cohomology, purity and formality with torsion coefficients.
\newblock {\em Journal of Topology}, 15(4):2270--2297, 2022.

\bibitem[CH24]{ciricicorrigendum}
Joana Cirici and Geoffroy Horel.
\newblock Corrigendum: {\'E}tale cohomology, purity and formality with torsion
  coefficients.
\newblock {\em Journal of Topology}, 17(2):e12348, 2024.

\bibitem[DCH21]{drummondhomotopy}
G.~C. Drummond-Cole and G.~Horel.
\newblock Homotopy transfer and formality.
\newblock {\em Ann. Inst. Fourier}, 71(5):2079--2116, 2021.

\bibitem[Del77]{sga4.5}
P.~Deligne.
\newblock {\em Cohomologie \'etale}, volume 569 of {\em Lecture Notes in
  Mathematics}.
\newblock Springer-Verlag, Berlin, 1977.
\newblock S\'eminaire de g\'eom\'etrie alg\'ebrique du Bois-Marie SGA
  $4\frac{1}{2}$.

\bibitem[Del80]{DeWeil}
P.~Deligne.
\newblock La conjecture de {W}eil. {II}.
\newblock {\em Inst. Hautes \'Etudes Sci. Publ. Math.}, (52):137--252, 1980.

\bibitem[DGMS75]{DGMS}
P.~Deligne, P.~Griffiths, J.~Morgan, and D.~Sullivan.
\newblock Real homotopy theory of {K}\"ahler manifolds.
\newblock {\em Invent. Math.}, 29(3):245--274, 1975.

\bibitem[Dot22]{Dotsenko}
V.~Dotsenko.
\newblock Homotopy invariants for via {K}oszul duality.
\newblock {\em Invent. Math.}, 228(1):77--106, 2022.

\bibitem[Dri90]{drinfeldquasitriangular}
V.~Drinfeld.
\newblock On quasitriangular quasi-{H}opf algebras and on a group that is
  closely connected with {${\rm Gal}(\overline{\bf Q}/{\bf Q})$}.
\newblock {\em Algebra i Analiz}, 2(4):149--181, 1990.

\bibitem[Dup16]{Dupont}
C.~Dupont.
\newblock Purity, {F}ormality, and {A}rrangement {C}omplements.
\newblock {\em Int. Math. Res. Not.}, (13):4132--4144, 2016.

\bibitem[EH92]{ElHa2}
M.~El~Haouari.
\newblock {$p$}-formalit\'e des espaces.
\newblock {\em J. Pure Appl. Algebra}, 78(1):27--47, 1992.

\bibitem[Eke86]{Ekedahl}
T.~Ekedahl.
\newblock Two examples of smooth projective varieties with nonzero {M}assey
  products.
\newblock In {\em Algebra, algebraic topology and their interactions
  ({S}tockholm, 1983)}, volume 1183 of {\em Lecture Notes in Math.}, pages
  128--132. Springer, Berlin, 1986.

\bibitem[{Fri}82]{friedlanderetale}
Eric~M. {Friedlander}.
\newblock {\em {Etale homotopy of simplicial schemes}}, volume 104.
\newblock Princeton University Press, Princeton, NJ, 1982.

\bibitem[GAV73]{sga4}
A.~Grothendieck, M.~Artin, and J.-L. Verdier.
\newblock Sga 4: Th{\'e}orie des topos et cohomologie {\'e}tale des
  sch{\'e}mas, tome 1--3.
\newblock {\em Lecture Notes in Math}, 269, 270, 305, 1973.

\bibitem[GNPR05]{santosmoduli}
F.~Guill{\'e}n, V.~Navarro, P.~Pascual, and A.~Roig.
\newblock Moduli spaces and formal operads.
\newblock {\em Duke Math. J.}, 129(2):291--335, 2005.

\bibitem[Hin15]{Hinichrectif}
V.~Hinich.
\newblock Rectification of algebras and modules.
\newblock {\em Doc. Math.}, 20, 2015.

\bibitem[HL88]{HaLe}
S.~Halperin and J.-M. Lemaire.
\newblock Notions of category in differential algebra.
\newblock In {\em Algebraic topology---rational homotopy ({L}ouvain-la-{N}euve,
  1986)}, volume 1318 of {\em Lecture Notes in Math.}, pages 138--154.
  Springer, Berlin, 1988.

\bibitem[Hor17]{horelprofinite}
G.~Horel.
\newblock Profinite completion of operads and the
  {G}rothendieck-{T}eichm\"{u}ller group.
\newblock {\em Adv. Math.}, 321:326--390, 2017.

\bibitem[Kee92]{Keel92}
Sean Keel.
\newblock Intersection theory of moduli space of stable {$n$}-pointed curves of
  genus zero.
\newblock {\em Trans. Amer. Math. Soc.}, 330(2):545--574, 1992.

\bibitem[Man06]{mandellcochains}
M.~A. Mandell.
\newblock Cochains and homotopy type.
\newblock {\em Publ. Math. Inst. Hautes \'{E}tudes Sci.}, (103):213--246, 2006.

\bibitem[Mat06]{Matei}
D.~Matei.
\newblock Massey products of complex hypersurface complements.
\newblock In {\em Singularity theory and its applications}, volume~43 of {\em
  Adv. Stud. Pure Math.}, pages 205--219. Math. Soc. Japan, Tokyo, 2006.

\bibitem[Mil80]{Milne}
J.~S. Milne.
\newblock {\em \'Etale cohomology}, volume~33 of {\em Princeton Mathematical
  Series}.
\newblock Princeton University Press, Princeton, N.J., 1980.

\bibitem[Mor78]{Mo}
J.~W. Morgan.
\newblock The algebraic topology of smooth algebraic varieties.
\newblock {\em Inst. Hautes \'Etudes Sci. Publ. Math.}, (48):137--204, 1978.

\bibitem[Pet14]{petersenminimal}
D.~Petersen.
\newblock Minimal models, {GT}-action and formality of the little disk operad.
\newblock {\em Selecta Math. (N.S.)}, 20(3):817--822, 2014.

\bibitem[Pet20]{Petersencompo}
D.~Petersen.
\newblock Cohomology of generalized configuration spaces.
\newblock {\em Compos. Math.}, 156(2):251--298, 2020.

\bibitem[Pet22]{PetersenRemark}
D.~Petersen.
\newblock A remark on singular cohomology and sheaf cohomology.
\newblock {\em Math. Scand.}, 128(2):229--238, 2022.

\bibitem[Qui11]{quickcontinuous}
G.~Quick.
\newblock Continuous group actions on profinite spaces.
\newblock {\em J. Pure Appl. Algebra}, 215(5):1024--1039, 2011.

\bibitem[RR15]{RoigSite}
B.~Rodr\'{\i}guez and A.~Roig.
\newblock Godement resolutions and sheaf homotopy theory.
\newblock {\em Collect. Math.}, 66(3):423--452, 2015.

\bibitem[Sal17]{Saleh}
B.~Saleh.
\newblock Noncommutative formality implies commutative and {L}ie formality.
\newblock {\em Algebr. Geom. Topol.}, 17(4):2523--2542, 2017.

\bibitem[Sal20]{Salvatore}
P.~Salvatore.
\newblock Non-formality of planar configuration spaces in characteristic 2.
\newblock {\em Int. Math. Res. Not.}, (10):3100--3129, 2020.

\bibitem[Sch97]{schnepsgrothendieck}
L.~Schneps.
\newblock The {G}rothendieck-{T}eichm{ü}ller group {GT}: a survey.
\newblock {\em Geometric Galois Actions, London Math. Soc. Lecture Notes},
  242:183--203, 1997.

\end{thebibliography}

\end{document}